\documentclass[11pt]{article}

\usepackage{booktabs}
\usepackage{relsize}
\usepackage{soul}
\usepackage{graphicx}
\usepackage{nameref}
\usepackage[shortlabels,inline]{enumitem}
\usepackage{empheq}
\usepackage[shortlabels]{enumitem}
\setlist[enumerate]{nosep}
\usepackage[doc]{optional}
\usepackage{xcolor}
\usepackage[colorlinks=true,
            linkcolor=refkey,
            urlcolor=lblue,
            citecolor=red]{hyperref}
\usepackage{float}
\usepackage{soul}
\usepackage{graphicx}
\definecolor{labelkey}{rgb}{0,0.08,0.45}
\definecolor{refkey}{rgb}{0,0.6,0.0}
\definecolor{Brown}{rgb}{0.45,0.0,0.05}
\definecolor{lime}{rgb}{0.00,0.8,0.0}
\definecolor{lblue}{rgb}{0.5,0.5,0.99}

 \usepackage{mathpazo}

\colorlet{hlcyan}{cyan!30}

\usepackage{stmaryrd}
\usepackage{amssymb}

\makeatletter
\def\namedlabel#1#2{\begingroup
   \def\@currentlabel{#2}%
   \label{#1}\endgroup
}
\makeatother

\newcommand{\seppfour}{\setlength{\itemsep}{-4pt}}

\newcommand{\bx}{\ensuremath{\mathbf{x}}}

\newcommand{\thalb}{\ensuremath{\tfrac{1}{2}}}
\newcommand{\menge}[2]{\big\{{#1}~\big |~{#2}\big\}}

\newcommand{\fenv}[1]%
{\ensuremath{\,\overrightarrow{\operatorname{env}}_{#1}}}
\newcommand{\benv}[1]%
{\ensuremath{\,\overleftarrow{\operatorname{env}}_{#1}}}

\newcommand{\scal}[2]{\left\langle{#1},{#2}  \right\rangle}

\newcommand{\RR}{\ensuremath{\mathbb R}}
\newcommand{\ii}{\ensuremath{\mathrm i}}
\newcommand{\RP}{\ensuremath{\mathbb{R}_+}}
\newcommand{\RPP}{\ensuremath{\mathbb{R}_{++}}}
\newcommand{\RM}{\ensuremath{\mathbb{R}_-}}

\newcommand{\dom}{\ensuremath{\operatorname{dom}}}

\newcommand{\prox}{\ensuremath{\operatorname{Prox}}}
\newcommand{\intdom}{\ensuremath{\operatorname{int}\operatorname{dom}}\,}

\newcommand{\ran}{\ensuremath{{\operatorname{ran}}\,}}

\newcommand{\Id}{\ensuremath{\operatorname{Id}}}

\newcommand{\minf}{\ensuremath{-\infty}}
\newcommand{\pinf}{\ensuremath{+\infty}}

\newcommand{\by}{\ensuremath{\mathbf{y}}}

\newcommand{\bz}{\ensuremath{\mathbf{z}}}

{\begin{list}{}{%
\settowidth{\labelwidth}{\textrm{#1~}}%
\setlength{\leftmargin}{\labelwidth+\labelsep}}}
{\end{list}}
\usepackage{amsthm}
\usepackage[capitalize,nameinlink]{cleveref}
\crefname{lemma}{Lemma}{Lemmas}
\crefname{equation}{}{equations}
\crefname{figure}{Figure}{Figures}
\crefname{chapter}{Appendix}{chapters}
\crefname{item}{}{items}
\crefname{enumi}{}{}
\crefname{enumii}{}{}
\newtheorem{theorem}{Theorem}[section]

\newtheorem{corollary}[theorem]{Corollary}

\newtheorem{proposition}[theorem]{Proposition}

\newtheorem{example}[theorem]{Example}
\newtheorem{ex}[theorem]{Example}
\newtheorem{fact}[theorem]{Fact}
\newtheorem{remark}[theorem]{Remark}

\providecommand{\RR}{\mathbb{R}}
\providecommand{\CC}{\mathbb{C}}

\providecommand{\ran}{\operatorname{ran}}

\providecommand{\dom}{\operatorname{dom}}

\providecommand{\sign}{\operatorname{sign}}

\providecommand{\epi}{\operatorname{epi}}

\providecommand{\Id}{\operatorname{{ Id}}}

\providecommand{\ran}{\operatorname{ran}}

\providecommand{\Id}{\operatorname{Id}}

\providecommand{\RR}{\mathbb{R}}

\definecolor{myblue}{rgb}{.8, .8, 1}
  \newcommand*\mybluebox[1]{%
    \colorbox{myblue}{\hspace{1em}#1\hspace{1em}}}

\allowdisplaybreaks 

\begin{document}

\title{\textsc{
Real roots of real cubics and optimization}}
\author{
Heinz H.\ Bauschke\thanks{
Mathematics, University
of British Columbia,
Kelowna, B.C.\ V1V~1V7, Canada. E-mail:
\texttt{heinz.bauschke@ubc.ca}.},~
Manish\ Krishan Lal\thanks{
                 Mathematics, University of British Columbia, Kelowna, B.C.\ V1V~1V7, Canada.
                 E-mail: \texttt{manish.krishanlal@ubc.ca}},~
         and Xianfu\ Wang\thanks{
                 Mathematics, University of British Columbia, Kelowna, B.C.\ V1V~1V7, Canada.
                 E-mail: \texttt{shawn.wang@ubc.ca}.}
                 }

\date{February 20, 2023}

\maketitle

\vskip 8mm

\begin{abstract} 
The solution of the cubic equation has a century-long history; however, the usual presentation is geared towards applications in algebra and is somewhat inconvenient to use in optimization where frequently the main interest lies in real roots. 

In this note, we present the roots of the cubic in a form that makes them convenient to use and we also focus on information on the location of the real roots. Armed with this, we provide several applications in optimization where we compute Fenchel conjugates, proximal mappings and projections.
\end{abstract}

{\small

}

\section{Introduction}

The history of solving cubic equations is rich and centuries old; see, e.g.,
Confalonieri's recent book \cite{Confa} on Cardano's work. Cubics do also appear in convex and nonconvex optimization. 
However, treatises on solving the cubic often focus on the general complex case making the results less useful to optimizers. 
The purpose of this note is two-fold. We present a largely self-contained derivation of the solution of the cubic with an emphasis on usefulness to practitioners. We do not claim novelty of these results; however, the presentation appears to be particularly convenient. We then turn to novel results. We show how the formulas can be used to compute Fenchel conjugates and proximal mappings of some convex functions. We also discuss projections on convex and nonconvex sets. 

\subsection{Outline of the paper}

The paper is organized as follows. 
In \cref{subsec:facts}, 
we collect some facts on polynomials. 
\cref{sec:depcubic} contains a self-contained treatment of the depressed cubic; in turn, this leads quickly to counterparts for the general cubic in \cref{sec:gencubic}. 
\cref{sec:convquar} concerns convex quartics --- we compute their
Fenchel conjugates and proximal mappings. 
In \cref{sec:alpha/x}, we present a formula for the proximal mapping of the convex reciprocal function. 
An explicit formula for the projection onto the epigraph of a parabola is provided in \cref{sec:projepipar}. 
In \cref{sec:missing}, we derive a formula for the projection of certain points onto a rectangular hyperbolic paraboloid. 
In the final \cref{sec:perspective}, we revisit the proximal mapping of the closure of a perspective function.

\subsection{Some facts}

\label{subsec:facts}

We now collect some properties of polynomials that are well known;
as a reference, we recommend \cite{RahSch}.

\begin{fact}
\label{f:reproots}
Let $f(x)$ be a nonconstant complex polynomial and let $r\in \CC$ such that $f(r)=0$. 
Then the multiplicity of $r$ is is the smallest integer $k$ such that 
the $k$th derivative at $r$ is nonzero: $f^{(k-1)}(r)=0$ and $f^{(k)}(r)\neq 0$. 
When $k=1$, $2$, or $3$, 
then we say that $r$ is a simple, double, or triple root, respectively. 
\end{fact}

\begin{fact} {\bf (Vieta)}
\label{f:Vieta}
Suppose $f(x)=ax^3+bx^2+cx+d$ is a cubic polynomial (i.e., $a\neq 0$) with complex coefficients.
If $r_1,r_2,r_3$ denote the (possibly repeated and complex) roots of $f$, then
\begin{subequations}
\label{e:Vieta}
\begin{align}
r_1+r_2+r_3&= -\frac{b}{a}\\
r_1r_2+r_1r_3 + r_2r_3 &= \frac{c}{a}\\
r_1r_2r_3 &= -\frac{d}{a}.
\end{align}
\end{subequations}
Conversely, if $r_1,r_2,r_3$ in $\mathbb{C}$ satisfy \cref{e:Vieta},
then they are the (possibly repeated) roots of $f$. 
\end{fact}

\begin{fact}
\label{f:cubicstart}
Suppose $f(x)=ax^3+bx^2+cx+d$ is a cubic polynomial (i.e., $a\neq 0$) with real coefficients. Then $f$ has three (possibly complex) roots (counting multiplicity).
More precisely, exactly one of the following holds:
\begin{enumerate}
\item $f$ has exactly one real root which either is simple (and the two remaining roots are nonreal simple roots and conjugate to each other) or is a triple root. 
\item $f$ has exactly two distinct real roots: one is simple and the other double.
\item $f$ has exactly three distinct simple real roots. 
\end{enumerate}
\end{fact}

\begin{remark}
We mention that the roots of a polynomial of a \emph{fixed} degree depend continuously on the coefficients --- see \cite[Theorem~1.3.1]{RahSch} 
for a precise statement and also the other results in \cite[Section~1.3]{RahSch}. 
\end{remark}

\section{The depressed cubic}

\label{sec:depcubic}

In this section, we study the \emph{depressed} cubic 
\begin{empheq}[box=\mybluebox]{equation}
g(z) := z^3+pz+q,
\quad\text{where}\;\;
p\in\RR \;\;\text{and}\;\; q\in\RR. 
\end{empheq}

\begin{theorem}
\label{t:roots}
We have 
\begin{equation}
g'(z)=3z^2+p
\;\;\text{and}\;\;
g''(z)=6z.
\end{equation}
Then 
$0$ is the only inflection point of $g$: 
$g$ is strictly concave on $\RM$ and $g$ is strictly convex on $\RP$.

Moreover, 
exactly one of the following cases occurs:
\begin{enumerate}
\item 
\label{t:roots1}
$p<0$: Set $z_\pm :=\pm\sqrt{-p/3}$. Then $z_-<z_+$, $z_\pm$ are two distinct simple roots of $g'$, 
$g$ is strictly increasing on $\left]\minf,z_-\right]$, 
$g$ is strictly decreasing on $[z_-,z_+]$,
$g$ is strictly increasing on $\left[z_+,\pinf\right[$. 
Moreover,
\begin{equation}
\label{e:221212c}
g(z_-)g(z_+)= 4\Delta, 
\quad\text{where}\;\;
\Delta := (p/3)^3+(q/2)^2,
\end{equation}
and this case trifurcates further as follows:

\begin{enumerate}
\item 
\label{t:roots1a} 
$\Delta>0$: Then $g$ has exactly one real root $r$. It is simple and
given by 
\begin{equation}
r := u_-+u_+,\quad
\text{where}\;\;
u_{\pm} := \sqrt[\mathlarger 3]{\frac{-q}{2}\pm \sqrt{\Delta}}. 
\end{equation}
The two remaining simple nonreal roots are 
\begin{equation}
-\thalb(u_-+u_+)\pm\ii\thalb\sqrt{3}(u_--u_+).
\end{equation}

\item 
\label{t:roots1b} 
$\Delta=0$: If $q>0$ (resp.\ $q<0$), then 
$2z_-$ (resp.\ $2z_+$) is a simple real root while $z_+$ (resp.\ $z_-$) is a double root.
Moreover, these cases can be combined into\footnote{Observe that this is the case when $\Delta\to 0^+$ in \cref{t:roots1a}.} 
\begin{equation}
\frac{3q}{p}= 2\sqrt[\mathlarger 3]{\frac{-q}{2}}\;\text{is a simple root of $g$}\;\;\text{and}\;\;
\frac{-3q}{2p}= -\sqrt[\mathlarger 3]{\frac{-q}{2}}\;\text{is a double root of $g$.}
\end{equation}

\item 
\label{t:roots1c} 
$\Delta<0$: 
Then $g$ has three simple real roots $r_-,r_0,r_+$ where
$r_-<z_-<r_0<z_+<r_+$.
Indeed, 
set
\begin{equation}
\theta := \arccos \frac{-q/2}{(-p/3)^{3/2}}, 
\end{equation}
which lies in 
$\left]0,\pi\right[$, and then define $z_0,z_1,z_2$ by 
\begin{equation}
z_k := 2(-p/3)^{1/2}\cos\Big(\frac{\theta+2k\pi}{3} \Big).
\end{equation}
Then $r_- = z_1$, $r_0=z_2$, and $r_+ = z_0$.
\end{enumerate}

\item 
\label{t:roots2}
$p=0$: Then $g'$ has a double root at $0$, and $g$ is strictly increasing on $\RR$.
The only real root is 
\begin{equation}
r := (-q)^{1/3}.
\end{equation}
If $q=0$, then $r$ is a triple root.
If $q\neq 0$, then $r$ is a simple root and the remaining nonreal simple roots 
are $-\thalb r \pm \ii\thalb\sqrt{3}r$.
\item 
\label{t:roots3}
$p>0$: Then $g'$ has no real root, $g$ is strictly increasing on $\RR$, and 
$g$ has exactly one real root $r$. It is simple and given by
\begin{equation}
r := u_-+u_+,\quad
\text{where}\;\;
u_{\pm} := \sqrt[\mathlarger 3]{\frac{-q}{2}\pm \sqrt{\Delta}} \;\;\text{and}\;\;
\Delta := (p/3)^3+(q/2)^2. 
\end{equation}
Once again, the two remaining simple nonreal roots are 
\begin{equation}
-\thalb(u_-+u_+)\pm\ii\thalb\sqrt{3}(u_--u_+).
\end{equation}
\end{enumerate}
\end{theorem}
\begin{proof}
Except for the formulas for the roots, all statements on $g$ follow from standard calculus.

\cref{t:roots1a}: 
Because $\Delta>0$ and $g$ is strictly decreasing on $[z_-,z_+]$, 
it follows from \cref{e:221212c}
that $g$ has the same sign on $[z_-,z_+]$ and so $g$ has no root in that interval.
Now $g$ is strictly increasing on $\left]\minf,z_-\right]$ and on 
$\left[z_+,\pinf\right[$; hence, 
$g$ has exactly one real root $r$ and it lies outside $[z_-,z_+]$.
Note that $r$ must be simple because the roots of $g'$ are $z_\mp$ and 
$r\neq z_\mp$. 
Note that $u_-<u_+$. 
Next,
$u_-^3u_+^3
= (q/2)^2-\Delta = -(p/3)^3
$
and so 
\begin{equation} 
\label{e:221212a}
u_-u_+ = -p/3.
\end{equation}
Also,
\begin{equation}
\label{e:221212b}
u_-^3+u_+^3 = \frac{-q}{2}-\sqrt{\Delta}+\frac{-q}{2}+\sqrt{\Delta}=-q. 
\end{equation}
Hence
\begin{align*}
g(r) &= r^3+pr+q\\
&=(u_-+u_+)^3+p(u_-+u_+)+q\\
&=u_-^3 + u_+^3 +3u_-u_+(u_-+u_+)+p(u_-+u_+)+q\\
&=\big(u_-^3 + u_+^3 \big) + (3u_-u_++p)(u_-+u_+)+q\\
&= -q + \big(3(-p/3)+p\big)(u_-+u_+)+q \tag{using \cref{e:221212a} and \cref{e:221212b}}\\
&= 0
\end{align*}
as claimed. 
Observe that we only need the properties \cref{e:221212a} and \cref{e:221212b}
about $u_-,u_+$ to conclude that $u_-+u_+$ is a root of $g$.
This observation leads us quickly to the two remaining complex roots:
First, denote the primitive 3rd root of unity by $\omega$, i.e., 
\begin{equation}
\label{e:prim3root}
\omega := \exp(2\pi \ii/3) = \cos(2\pi/3)+\ii\sin(2\pi/3)=-\thalb +\ii\thalb\sqrt{3}. 
\end{equation}
Then $\omega^2 = \overline{\omega} = -\thalb -\ii\thalb\sqrt{3}$ 
and $\omega^3 = \overline{\omega}^3 = 1$. 
Now set
\begin{equation*}
v_- := \omega u_- \;\;\text{and}\;\; v_+ := \omega^2u_+ = \overline{\omega}u_+.
\end{equation*}
Then $v_-v_+(\omega u_-)=(\omega^2u_+)=\omega^3u_-u_+=u_-u_+=-p/3$ 
by \cref{e:221212a}, and 
$v_-^3+v_+^3=(\omega u_-)^3+(\omega^2u_+)^3 =  \omega^3u_-^3+\omega^6u_+^3
=u_-^3+u_+^3=-q$ by \cref{e:221212a}.
Hence 
\begin{align*}
v_-+v_+ 
&= 
\omega u_- + \overline{\omega}u_+\\
&=\big(-\thalb +\ii\thalb\sqrt{3}\big)u_- + \big(-\thalb -\ii\thalb\sqrt{3}\big)u_+\\
&= -\thalb(u_-+u_+)+\ii\thalb\sqrt{3}(u_--u_+) 
\end{align*}
and its conjugate are the remaining simple complex roots of $g$.

\cref{t:roots1b}:
From \cref{e:221212c}, it follows that $z_-$ or $z_+$ is a root of $g$.
In view of \cref{f:reproots} and $g'(z_-)=g'(z_+)$, it follows that one of 
$z_-,z_+$ is at least a double root, but not both; moreover, it cannot be a triple root 
because $0$ is the only root of $g''$ and $z_-<0<z_+$. 
Hence exactly one of $z_-,z_+$ is a double root. 
To verify the remaining parts, we first define
\begin{equation*}
r_1 := \frac{3q}{p}\;\;\text{and}\;\;
r_2 := \frac{-3q}{2p}. 
\end{equation*}
Because $\Delta=0$, it follows that $4p^3+27q^2=0$.
Hence
\begin{align*}
g(r_1)
&= r_1^3+pr_1+q
= \frac{27q^3}{p^3}+\frac{3pq}{p}+q
= \frac{27q^3}{p^3}+4q
= \frac{q}{p^3}\big(27q^2+4p^3 \big)
= 0
\end{align*}
and
\begin{align*}
g(r_2)
&= r_2^3+pr_2+q
= \frac{-27q^3}{8p^3}+\frac{-3pq}{2p}+q
= \frac{-27q^3}{8p^3}-\frac{q}{2}
= \frac{-q}{8p^3}\big(27q^2+4p^3 \big)
= 0.
\end{align*}
The assumption that $\Delta=0$ readily yields
\begin{equation*}
p = \frac{-3^{1/3}q^{2/3}}{2^{2/3}}\;\;\text{and}\;\;
|q| = \frac{2(-p)^{3/2}}{3^{3/2}}.
\end{equation*}
Hence
\begin{equation*}
r_1 = 3q p^{-1}
= 3q (-1)3^{-1/3}q^{-2/3}2^{2/3} = 2^{2/3}(-q)^{1/3}
\end{equation*}
and
\begin{equation*}
r_2 = -3q 2^{-1}p^{-1}
= -3q 2^{-1}(-1)3^{-1/3}q^{-2/3}2^{2/3} = -2^{-1/3}(-q)^{1/3}
\end{equation*}
as claimed. 

If $q>0$, then 
\begin{equation*}
r_1 = \frac{3q}{p} = \frac{3\cdot 2(-p)^{3/2}}{3^{3/2}p}=-2(-p/3)^{1/2}=2z_-
\end{equation*}
and
\begin{equation*}
r_2 = \frac{-3q}{2p}=-\thalb\frac{3q}{p}=-\thalb r_1 = -\thalb 2z_-=z_+.
\end{equation*}
Similarly, if $q<0$, then 
$r_1=2z_+$ and $r_2=z_-$.

No matter the sign of $q$, we have $r_2\in\{z_-,z_+\}$ and thus $g'(r_2)=0$, i.e., 
$r_2$ is the double root. 

\cref{t:roots1c}: 
In view of \cref{e:221212c}, $g(z_-)$ and $g(z_+)$ have opposite signs.
Because $g$ is strictly decreasing on $[z_-,z_+]$, it follows that 
$g(z_-)>0>g(z_+)$. 
Hence there is at least on real root $r_0$ in $\left]z_-,z_+\right[$.
On the other hand, $g$ is strictly increasing on $\left]\minf,z_-\right]$ 
and on $\left[z_+,\pinf\right[$ which yields further roots $r_-$ and $r_+$ as announced. Having now three real roots, they must all be simple.

Next, note that $\Delta<0$ 
$\Leftrightarrow$
$0\leq (q/2)^2<-(p/3)^3 = (-p/3)^3$
$\Leftrightarrow$
$0\leq (q/2)^2/(-p/3)^3<1$
$\Leftrightarrow$
$0\leq (|q|/2)/(-p/3)^{3/2}<1$
$\Leftrightarrow$
$-1<(-q/2)/(-p/3)^{3/2}<1$.
It follows that 
\begin{equation}
\label{e:221213b}
\theta = \arccos \frac{-q/2}{(-p/3)^{3/2}} \in \left]0,\pi\right[
\end{equation}
as claimed.
For convenience, we set, for $k\in\{0,1,2\}$, 
\begin{equation}
\theta_k := \frac{\theta+2k\pi}{3};
\quad\text{hence,}\;\;
z_k = 2(-p/3)^{1/2}\cos(\theta_k). 
\end{equation}
Recall that $0<\theta<\pi$, which allows us to draw three conclusions:
\begin{subequations}
\begin{align}
0<\theta_0=\theta/3<\pi/3 
&\Rightarrow 
1>\cos(\theta_0)=\cos(\theta/3)>1/2;\\
2\pi/3<\theta_1=(\theta+2\pi)/3<\pi
&\Rightarrow 
-1/2 > \cos(\theta_1)=\cos((\theta+2\pi)/3) > -1;\\
4\pi/3<\theta_2 = (\theta+4\pi)/3<5\pi/3 
&\Rightarrow 
-1/2 < \cos(\theta_2)=\cos((\theta+2\pi)/3) < 1/2.
\end{align}
\end{subequations}
Hence $\cos(\theta_1)<\cos(\theta_2)<\cos(\theta_0)$ and thus 
\begin{equation}
z_1<z_2<z_0. 
\end{equation}
All we need to do is to verify that each $z_k$ is actually a root of $g$.
To this end, observe first that 
the triple-angle formula for the cosine 
(see, e.g., \cite[Formula~4.3.28 on page~72]{AS}) yields
\begin{subequations}
\label{e:221213a}
\begin{align}
\cos^3(\theta_k)
&=\frac{3\cos(\theta_k)+\cos(3\theta_k)}{4}
= \frac{3\cos(\theta_k)+\cos(\theta+2k\pi)}{4}\\
&= \frac{3\cos(\theta_k)+\cos(\theta)}{4}. 
\end{align}
\end{subequations}
Then 
\begin{align*}
g(z_k) &= 
z_k^3+pz_k +q\\
&=
8(-p/3)^{3/2}\cos^3(\theta_k)+p2(-p/3)^{1/2}\cos(\theta_k)+q\\
&= 
2(-p/3)^{3/2}\big(3\cos(\theta_k)+\cos(\theta) \big)+2(-p/3)^{1/2}p\cos(\theta_k)+q
\tag{using \cref{e:221213a}}\\
&=2(-p/3)^{1/2}\cos(\theta_k)\big(3(-p/3)+p\big)+2(-p/3)^{3/2}\cos(\theta)+q\\
&=2(-p/3)^{3/2}\cos(\theta)+q\\
&= 
2(-p/3)^{3/2}\frac{-q/2}{(-p/3)^{3/2}}
+ q\tag{using \cref{e:221213b}}\\
&= 0,
\end{align*}
and this completes the proof for this case.

\cref{t:roots2}: If $q=0$, then $g(z)=z^3$ so $z=0$ is the only root of $g$ and it is of multiplicity $3$. Thus we assume that $q\neq 0$. 
Then $g(z)=0$ $\Leftrightarrow$ $z^3+q =0$
$\Leftrightarrow$ $z^3=-q$
$\Rightarrow$ $z=(-q)^{1/3}\neq 0$.  
Because $g$ is strictly increasing on $\RR$, $r:= (-q)^{1/3}$ is the only real root of $g$. 
Because $g'$ has only one real root, namely $0$, it follows that 
$g'(r)\neq 0$ and so $r$ is a simple root. 
Denoting again by $\omega$ the primitive 3rd root of unity (see \cref{e:prim3root}), 
it is clear
that the remaining complex (simple) roots
are $\omega r$ and $\overline{\omega}r$ as claimed.

\cref{t:roots3}: 
Note that $\Delta \geq (p/3)^3>0$ because $p>0$. 
The fact that $r$ is a root is shown exactly as in \cref{t:roots1a}.
It is simple because $g'$ has no real roots,
and $r$ is unique because $g$ is strictly increasing.
The complex roots are derived exactly as in \cref{t:roots1a}. 
\end{proof}

We now provide a conise version of \cref{t:roots}:

\begin{corollary} {\bf (trichotomy)}
\label{c:roots}
Set $\Delta := (p/3)^3+(q/2)^2$.
Then exactly one of the following holds:
\begin{enumerate}
\item $p=0$ or $\Delta>0$: Then $g$ has exactly one real root and it is given by
\begin{equation}
\sqrt[\mathlarger 3]{\frac{-q}{2}+ \sqrt{\Delta}}
+
\sqrt[\mathlarger 3]{\frac{-q}{2}-\sqrt{\Delta}}. 
\end{equation}
\item $p<0$ and $\Delta=0$: 
Then $g$ has exactly two real roots which are given by
\begin{equation}
\frac{3q}{p}= 2\sqrt[\mathlarger 3]{\frac{-q}{2}}
\;\;\text{and}\;\;
\frac{-3q}{2p}= -\sqrt[\mathlarger 3]{\frac{-q}{2}}. 
\end{equation}
\item $\Delta<0$:
Then $g$ has exactly three real roots $z_0,z_1,z_2$ which are given by 
\begin{equation}
z_k := 2(-p/3)^{1/2}\cos\Big(\frac{\theta+2k\pi}{3} \Big), 
\quad\text{where}\;\;
\theta := \arccos \frac{-q/2}{(-p/3)^{3/2}},
\end{equation}
and where $z_1<z_2<z_0$.
\end{enumerate}
\end{corollary}

\section{The general cubic}

\label{sec:gencubic}
In this section, we turn to the general cubic
\begin{empheq}[box=\mybluebox]{equation}
\label{e:gencubic}
f(x) := ax^3+bx^2+cx+d, 
\quad\text{where}\;\;
a,b,c,d\;\text{are in}\;\RR\;\text{and}\; a>0.
\end{empheq}
(The case $a<0$ is treated similarly.)
Note that $f''(x)=6ax+2b$ has exactly one zero, namely 
\begin{equation}
\label{e:defx0}
x_0 := \frac{-b}{3a}.
\end{equation}
The change of variables
\begin{equation}
\label{e:charvar}
x = z+x_0 
\end{equation}
leads to the well known depressed cubic
\begin{equation}
\label{e:defpq}
g(z) := z^3 + pz+q, \;\;\text{where}\;\; p := \frac{3ac-b^2}{3a^2}
\;\;\text{and}\;\; q := \frac{27a^2d+2b^3-9abc}{27a^3}
\end{equation}
which we reviewed in \cref{sec:depcubic}. 
Here $ag(z)=f(x)=f(z+x_0)$ so 
the roots of $g$ are precisely those of $f$, translated by $x_0$:
\begin{equation}
\text{$x$ is a root of $f$ $\Leftrightarrow$ $x-x_0$ is a root of $g$.}
\end{equation}
So all we need to do is find the roots of $g$, and then add $x_0$ to them, to obtain
the roots of $f$. Because the change of variables \cref{e:charvar} is linear,
multiplicity of the roots are preserved. 
Translating some of the results from \cref{t:roots} for $g$ to $f$ gives the following:

\begin{theorem}
\label{t:genroots}
$f$ 
is strictly concave on $\left]\minf,x_0\right]$ and 
is strictly convex on $\left[x_0,\pinf\right[$, where $x_0$ is the unique inflection point of $f$ defined in \cref{e:defx0}. 
Recall the definitions of $p,q$ from \cref{e:defpq} and also set 
\begin{equation}
\Delta := (p/3)^3+(q/2)^2 = \frac{(3ac-b^2)^3}{(9a^2)^3}
+ \frac{(27a^2d+2b^3-9abc)^2}{(54a^3)^2}.
\end{equation}
Then exactly one of the following cases occurs:
\begin{enumerate}
\item 
\label{t:genroots1}
\fbox{$b^2>3ac \Leftrightarrow p<0$}\,: Set $x_\pm :=(-b\pm\sqrt{b^2-3ac})/(3a)$. 
Then $x_\pm$ are two distinct simple roots of $f'$, 
$f$ is strictly increasing on $\left]\minf,x_-\right]$, 
$f$ is strictly decreasing on $[x_-,x_+]$,
$f$ is strictly increasing on $\left[x_+,\pinf\right[$. 
This case trifurcates further as follows:

\begin{enumerate}
\item 
\label{t:genroots1a} 
\fbox{$\Delta>0$}\,: Then $f$ has exactly one real root; moreover, it
is simple and given by
\begin{equation}
x_0+u_-+u_+,\quad
\text{where}\;\;
u_{\pm} := \sqrt[\mathlarger 3]{\frac{-q}{2}\pm \sqrt{\Delta}}. 
\end{equation}
The two remaining simple nonreal roots are 
$x_0-\thalb(u_-+u_+)\pm\ii\thalb\sqrt{3}(u_--u_+)$.

\item 
\label{t:genroots1b} 
\fbox{$\Delta=0$}\,: 
Then $f$ has two distinct real roots: The simple root is
\begin{equation}
x_0+\frac{3q}{p} = 
x_0+2\sqrt[\mathlarger 3]{\frac{-q}{2}} = 
\frac{4 a b c -b^3 -9 a^{2} d}{a{\left(b^{2} - 3  a c\right)} } 
\end{equation}
and the double root is
\begin{equation}
x_0-\frac{3q}{2p} = x_0 -\sqrt[\mathlarger 3]{\frac{-q}{2}} = 
\frac{9ad-bc}{2(b^{2} - 3  a c)}.
\end{equation}
\item 
\label{t:genroots1c} 
\fbox{$\Delta<0$}\,: 
Then $f$ has three simple real roots $r_-,r_0,r_+$ where
$r_-<x_-<r_0<x_+<r_+$.
Indeed, 
set
\begin{equation}
\theta := \arccos \frac{-q/2}{(-p/3)^{3/2}}, 
\end{equation}
which lies in 
$\left]0,\pi\right[$, and then define $y_0,y_1,y_2$ by 
\begin{equation}
y_k := x_0+2(-p/3)^{1/2}\cos\Big(\frac{\theta+2k\pi}{3} \Big).
\end{equation}
Then $r_- = y_1$, $r_0=y_2$, and $r_+ = y_0$.
\end{enumerate}
\item 
\label{t:genroots2}
\fbox{$b^2=3ac \Leftrightarrow p=0$}\,: 
Then $f$ is strictly increasing on $\RR$ and its
only real root is 
\begin{equation}
r := x_0 + (-q)^{1/3}.
\end{equation}
If $q=0$, then $r$ is a triple root.
If $q\neq 0$, then $r$ is a simple root and the remaining nonreal simple roots 
are $x_0-\thalb (-q)^{1/3} \pm \ii\thalb\sqrt{3}(-q)^{1/3}$.
\item 
\label{t:genroots3}
\fbox{$b^2<3ac \Leftrightarrow p>0$}\,: 
Then $f$ is strictly increasing on $\RR$, and 
$f$ has exactly one real root; moreover, it is simple and given by
\begin{equation}
x_0+u_-+u_+,\quad
\text{where}\;\;
u_{\pm} := \sqrt[\mathlarger 3]{\frac{-q}{2}\pm \sqrt{\Delta}}.
\end{equation}
The two remaining simple nonreal roots are 
$x_0-\thalb(u_-+u_+)\pm\ii\thalb\sqrt{3}(u_--u_+)$.
\end{enumerate}
\end{theorem}

In turn, \cref{c:roots} turns into 

\begin{corollary}
\label{c:genroots}
Recall \cref{e:defx0} and \cref{e:defpq}, and
set 
\begin{equation}
\Delta := (p/3)^3+(q/2)^2 = \frac{(3ac-b^2)^3}{(9a^2)^3}
+ \frac{(27a^2d+2b^3-9abc)^2}{(54a^3)^2}
\end{equation}
Then exactly one of the following holds:
\begin{enumerate}
\item 
\label{c:genroots1}
\fbox{$b^2=3ac$ or $\Delta>0$}\,: 
Then $f$ has exactly one real root and it is given by
\begin{equation}
x_0+ \sqrt[\mathlarger 3]{\frac{-q}{2}+ \sqrt{\Delta}}
+
\sqrt[\mathlarger 3]{\frac{-q}{2}-\sqrt{\Delta}}. 
\end{equation}
\item 
\label{c:genroots2}
\fbox{$b^2>3ac$ and $\Delta=0$}\,: 
Then $f$ has exactly two real roots which are given by
\begin{equation}
x_0+\frac{3q}{p}= x_0+2\sqrt[\mathlarger 3]{\frac{-q}{2}}
\;\;\text{and}\;\;
x_0+\frac{-3q}{2p}= x_0-\sqrt[\mathlarger 3]{\frac{-q}{2}}. 
\end{equation}
\item 
\label{c:genroots3}
\fbox{$\Delta<0$}\,:
Then $f$ has exactly three real (simple) roots  $r_0,r_1,r_2$, where
\begin{equation}
r_k := x_0+2(-p/3)^{1/2}\cos\Big(\frac{\theta+2k\pi}{3} \Big), 
\quad\text{}\;\;
\theta := \arccos \frac{-q/2}{(-p/3)^{3/2}},
\end{equation}
and $r_1<r_2<r_0$. 
\end{enumerate}
\end{corollary}

\section{Convex Analysis of the general quartic}

\label{sec:convquar}

In this section, we study the function
\begin{empheq}[box=\mybluebox]{equation}
\label{e:genquartic}
h(x) := \alpha x^4 + \beta x^3+ \gamma x^2+ \delta x + \varepsilon, \quad
\text{where $\alpha,\beta,\gamma,\delta,\varepsilon$ are in $\RR$ with $\alpha\neq 0$.}
\end{empheq}
We start by characterizing convexity.

\begin{proposition} {\bf (convexity)}
\label{p:convexquartic}
The general quartic \cref{e:genquartic} is convex
if and only if
\begin{equation}
\label{e:221214a}
\alpha>0 \;\;\text{and}\;\; 8\alpha\gamma\geq 3\beta^2. 
\end{equation}
\end{proposition}
\begin{proof}
Note that 
$h'(x)=4\alpha x^3 + 3\beta x^2 + 2\gamma x + \delta$ and, also completing the square, 
\begin{align}
h''(x)&=12 \alpha x^2 + 6\beta x + 2\gamma
= \frac{3}{4}\alpha\Big(4x+\frac{\beta}{\alpha}\Big)^2 + 2\gamma - \frac{3\beta^2}{4\alpha}. 
\end{align}
Hence $h''\geq 0$
$\Leftrightarrow$
[$\alpha>0$ and $2\gamma \geq 3\beta^2/(4\alpha)$]
$\Leftrightarrow$  \cref{e:221214a}. 
(For further information on deciding nonnegativity of polynomials, 
see \cite[Section~3.1.3]{BPT}.)
\end{proof}

Having characterization convexity, we shall assume this condition for the remainder of this section:
\begin{empheq}[box=\mybluebox]{equation}
\label{e:convexquartic}
\text{$h$ is convex, i.e.,\;\;}
\alpha >0 \;\text{and}\; 8\alpha\gamma\geq 3\beta^2.
\end{empheq}

\begin{proposition} {\bf (Fenchel conjugate)}
\label{p:Fenchelquartic}
Recall our assumptions \cref{e:genquartic} and \cref{e:convexquartic}. 
Let $y\in \RR$.
Then 
\begin{equation}
h^*(y) = yx_y-h(x_y), 
\end{equation}
where 
$p:=(8\alpha\gamma-3\beta^2)/(16\alpha^2)\geq 0$,
$q:=(8\alpha^2(\delta-y)+\beta^3-4\alpha\beta\gamma)/(32\alpha^3)$,
$\Delta := (p/3)^2+(q/2)^2\geq 0$, and 
\begin{equation}
\label{e:theo1}
x_y := 
-\frac{\beta}{4\alpha}+
\sqrt[\mathlarger 3]{\frac{-q}{2}+\sqrt{\Delta}}
+
\sqrt[\mathlarger 3]{\frac{-q}{2}-\sqrt{\Delta}}. 
\end{equation}
\end{proposition}
\begin{proof}
Because $h$ is supercoercive, it follows from 
\cite[Proposition~14.15]{BC2017} that $\dom h^*=\RR$.
Combining with the differentiability of $h$, it follows that
$y\in \intdom h^*\subseteq \dom \partial h^* = \ran \partial h = \ran h'$. 
However, if $h'(x)=y$, then $h^*(y)=xy-h(x)$ and we have found the conjugate.
It remains to solve $h'(x)=y$, i.e., 
\begin{equation}
\label{e:221220a}
4\alpha x^3+3\beta x^2+2\gamma x+\delta-y=0.
\end{equation}
So we set 
\begin{equation*}
f(x) := ax^3+bx^2+cx+d,
\quad\text{where}\;\;
a := 4\alpha,
\;\;
b := 3\beta,
\;\;
c := 2\gamma,
\;\;
d := \delta-y. 
\end{equation*}
To solve \cref{e:221220a}, i.e., $f(x)=0$, we first note that
\begin{align*}
p 
&= 
\frac{3ac-b^2}{3a^2}
= \frac{3(4\alpha)(2\gamma)-(3\beta)^2}{3(4\alpha)^2}
=
\frac{8\alpha\gamma-3\beta^2}{16\alpha^2}
\geq 0,
\end{align*}
where the inequality follows from \cref{e:convexquartic}.
Next, 
\begin{align*}
q 
&=
\frac{3^3a^2d+2b^3-3^2abc}{(3a)^3}
= \frac{3^34^2\alpha^2(\delta-y)+2(3^3\beta^3)-3^2(4\alpha)(3\beta)(2\gamma)}{3^34^3\alpha^3}\\
&=
\frac{8\alpha^2(\delta-y)+\beta^3-4\alpha\beta\gamma}{32\alpha^3} 
\end{align*}
and 
\begin{equation*}
\Delta = (p/3)^3 + (q/2)^2\geq 0,
\end{equation*}
where the inequality follows because $p\geq 0$. 
Then $-b/(3a)=-\beta/(4\alpha)$ and now 
\cref{c:genroots}\cref{c:genroots1} yields the unique solution of $f(x)=0$ as 
\cref{e:theo1}. 
\end{proof}

\begin{proposition} {\bf (proximal mapping)}
\label{p:proxquartic}
Recall our assumptions \cref{e:genquartic} and \cref{e:convexquartic}. 
Let $y\in \RR$.
Then 
\begin{equation}
\label{e:proxquartic}
\prox_h(y) = -\frac{\beta}{4\alpha}+\sqrt[\mathlarger 3]{\frac{-q}{2}+ \sqrt{\Delta}}
+
\sqrt[\mathlarger 3]{\frac{-q}{2}-\sqrt{\Delta}}, 
\end{equation}
where 
\begin{align}
\label{e:proxquarticpq}
p := \frac{4\alpha(1+2\gamma)-3\beta^2}{16\alpha^2},
\quad 
q := \frac{8\alpha^2(\delta-y)+\beta^3-2\alpha\beta(1+2\gamma)}{32\alpha^3},
\end{align}
and 
$\Delta := (p/3)^3+(q/2)^2\geq 0$.
\end{proposition}
\begin{proof}
Because $h$ is differentiable and full domain,
it follows that $\prox_h(y)$ is the \emph{unique} solution $x$
of the equation $h'(x)+x-y=0$. 
The proof thus proceeds analogously to that of 
\cref{p:Fenchelquartic} --- the only difference is we must solve
\begin{equation*}
f(x) := ax^3+bx^2+cx+d,
\quad\text{where}\;\;
a := 4\alpha,
\;\;
b := 3\beta,
\;\;
c := 2\gamma+1,
\;\;
d := \delta-y. 
\end{equation*}
(The only difference is that $c=2\gamma+1$ rather than $2\gamma$ due to the additional term ``$+x$''.)
Thus we know \emph{a priori} that the resulting cubic must have a unique real solution. 
We now have
\begin{align*}
\label{e:221214b}
0 &< \frac{1}{4\alpha} \leq 
\frac{1}{4\alpha}+\frac{8\alpha\gamma-3\beta^2}{16\alpha^2}
= \frac{4\alpha(1+2\gamma)-3\beta^2}{16\alpha^2}=p 
= \frac{12\alpha(1+2\gamma)-9\beta^2}{48\alpha^2}\\
&= \frac{3ac-b^2}{3a^2},
\end{align*}
which is our usual $p$ from discussing roots of the cubic $f$.
Similarly, the $q$ defined here is the same as the usual $q$ for $f(x)$ (see \cref{e:defpq}). 
Finally, the formula for $x=\prox_h(y)$ now follows from 
\cref{c:genroots}\cref{c:genroots1}. 
\end{proof}

\begin{example}
\label{ex:geoquart}
Suppose that 
\begin{equation}
h(x)=x^4+x^3+x^2+x+1,
\end{equation}
and let $y\in\RR$. 
Then $h$ is convex and
\begin{subequations}
\begin{align}
\label{e:221220b}
h^*(y)&=yx_y-h(x_y),\quad\text{where}\\
x_y &= 
-\frac{1}{4}+
\frac{1}{2}\sqrt[\mathlarger 3]{y-\tfrac{5}{8}+\sqrt{(y-\tfrac{5}{8})^2+(\tfrac{5}{4})^3}}
+
\frac{1}{2}\sqrt[\mathlarger 3]{y-\tfrac{5}{8}-\sqrt{(y-\tfrac{5}{8})^2+(\tfrac{5}{4})^3}}.
\end{align}
\end{subequations}
Moreover, 
\begin{equation}
\label{e:221215b}
\prox_h(y) = -\frac{1}{4}+
\frac{1}{2}\sqrt[\mathlarger 3]{y-\tfrac{3}{8}+{\sqrt{(y-\tfrac{3}{8})^2+(\tfrac{3}{4})^3}}}
+\frac{1}{2}\sqrt[\mathlarger 3]{y-\tfrac{3}{8}-{\sqrt{(y-\tfrac{3}{8})^2+(\tfrac{3}{4})^3}}}. 
\end{equation}
See \cref{fig:figure3} for a visualization. 
\end{example}
\begin{proof}
Note that $h$ fits the pattern of \cref{e:genquartic} with 
$\alpha=\beta=\gamma=\delta=\varepsilon=1$.
The characterization of convexity presented in
\cref{p:convexquartic} turns into $1>0$ and $8\geq 3$ which are both obviously true. Hence $h$ is convex.

To compute the Fenchel conjugate, we apply \cref{p:Fenchelquartic} and get 
$p=5/16$, $q=(5-8y)/32=-(y-5/8)/4$, and
$\Delta = 5^3/16^3 + (y-5/8)^2/8^2$. Then $-q/2=(y-5/8)/8$.
Hence 
\cref{e:theo1} turns into
\begin{align*}
x_y &=
-\frac{1}{4}+
\sqrt[\mathlarger 3]{\frac{y-5/8}{8}+\sqrt{\frac{(y-5/8)^2}{8^2}+\frac{5^3}{16^3}}}
+
\sqrt[\mathlarger 3]{\frac{y-5/8}{8}-\sqrt{\frac{(y-5/8)^2}{8^2}+\frac{5^3}{16^3}}}\\
&=
-\frac{1}{4}+
\sqrt[\mathlarger 3]{\frac{y-5/8}{8}+\sqrt{\frac{(y-5/8)^2}{8^2}+\frac{5^3}{8^2\cdot 4^3}}}
+
\sqrt[\mathlarger 3]{\frac{y-5/8}{8}-\sqrt{\frac{(y-5/8)^2}{8^2}+\frac{5^3}{8^2\cdot 4^3}}},
\end{align*}
which simplifies to \cref{e:221220b}.

To compute $\prox_h(y)$, we utilize 
\cref{p:proxquartic}. Obtaining fresh values for $p,q,\Delta$, we have this time $p=9/16$, $q=(3-8y)/32$, 
$\Delta = ((8y-3)^2+27)/4096 = ((8y-3)^2+3^3)/64^2$.
Hence 
$-q/2 = (8y-3)/64$ and 
$\sqrt{\Delta}=\sqrt{(8y-3)^2+3^3}/64$.
It follows that 
\begin{align*}
\sqrt[\mathlarger 3]{\frac{-q}{2}\pm\sqrt{\Delta}}
&= 
\sqrt[\mathlarger 3]{\frac{8y-3}{64}\pm\frac{\sqrt{(8y-3)^2+3^3}}{64}}
= \frac{1}{4}\sqrt[\mathlarger 3]{8y-3\pm{\sqrt{(8y-3)^2+3^3}}}\\
&= \frac{1}{2}\sqrt[\mathlarger 3]{y-\tfrac{3}{8}\pm{\sqrt{(y-\tfrac{3}{8})^2+(\tfrac{3}{4})^3}}}
\end{align*}
This, $-\beta/(4\alpha)=-1/4$, and \cref{e:proxquartic}
now yields \cref{e:221215b}. 
\end{proof}

\begin{figure}[htp]

\centering
\includegraphics[width=.31\textwidth]{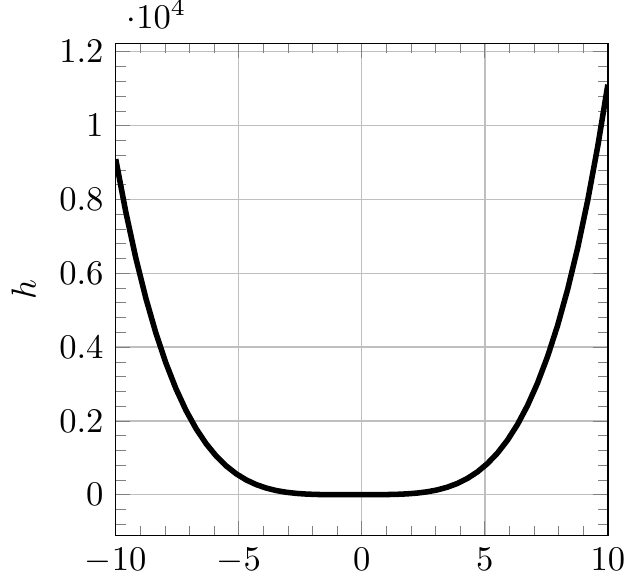}\hfill
\includegraphics[width=.32\textwidth]{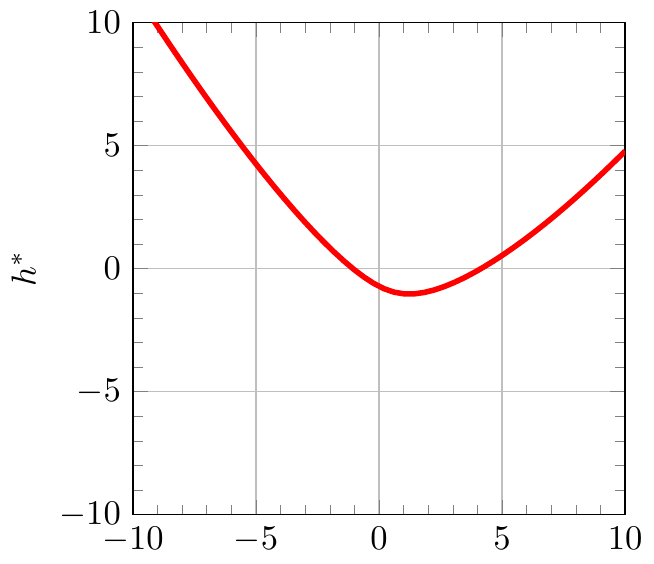}\hfill
\includegraphics[width=.325\textwidth]{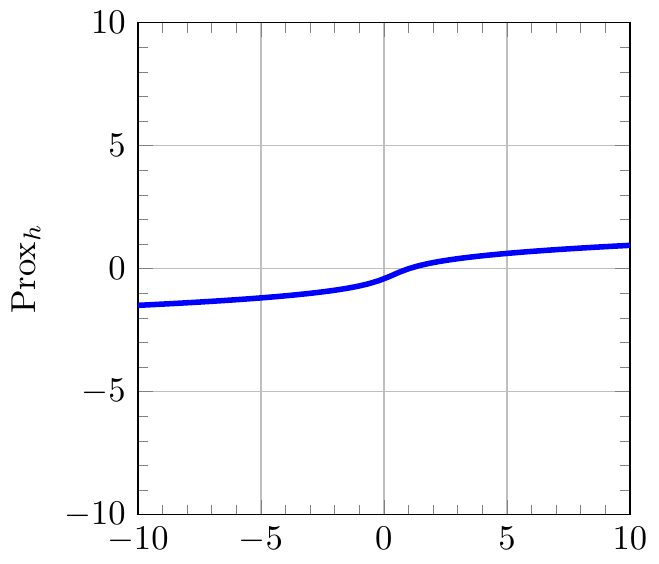}

\caption{A visualization of \cref{ex:geoquart}. 
Depicted are $h$ (left), its conjugate $h^*$ (middle), and the proximal mapping $\prox_h$ (right).}
\label{fig:figure3}
\end{figure}

\begin{example}
Suppose that 
\begin{equation}
h(x) = \alpha x^4, \quad\text{where $\alpha>0$,}
\end{equation}
and let $y\in\RR$. 
Then 
\begin{equation}
\label{e:alpha4*}
h^*(y)=\frac{3}{4(4\alpha)^{1/3}}y^{4/3}
\end{equation}
and 
\begin{equation}
\label{e:221215a}
\prox_h(y)=\frac{1}{2}\sqrt[\mathlarger 3]{\frac{y}{\alpha}+ \sqrt{\frac{1+27\alpha y^2}{27\alpha^3}}}
+
\frac{1}{2}\sqrt[\mathlarger 3]{\frac{y}{\alpha}- \sqrt{\frac{1+27\alpha y^2}{27\alpha^3}}}. 
\end{equation}
\end{example}
\begin{proof}
Note that $h$ fits the pattern of \cref{e:genquartic} with 
$\beta=\gamma=\delta=\varepsilon=0$.
The characterization of convexity presented in
\cref{p:convexquartic} turns into $\alpha>0$ and $0\geq 0$ which are both obviously true. Hence $h$ is convex.

We start by computing the Fenchel conjugate of $h$ using \cref{p:Fenchelquartic}.
We have $p=0$, $q=-y/(4\alpha)$, $\Delta = (y/(8\alpha))^2$, 
and $-\beta/(4\alpha)=0$. 
Hence $-q/2=y/(8\alpha)$ and $\sqrt{\Delta}=|y|/(8\alpha)$ which imply
$\sqrt[3]{(-q/2)\pm\sqrt{\Delta}} =\sqrt[3]{y/(8\alpha)\pm |y|/(8\alpha) }
= \sqrt[3]{\max\{0,y/(4\alpha)\}}$ or 
$\sqrt[3]{\min\{0,y/(4\alpha)\}}$. 
Using \cref{e:theo1}, we get
\begin{align*}
x_y &= 
\sqrt[\mathlarger 3]{\max\{0,y/(4\alpha)\}}
+\sqrt[\mathlarger 3]{\min\{0,y/(4\alpha)\}}
= \sqrt[3]{y/(4\alpha)}. 
\end{align*}
Using \cref{p:Fenchelquartic}, we obtain
\begin{align*}
h^*(y)
&=yx_y-h(x_y)
=yy^{1/3}/(4\alpha)^{1/3}
-\alpha y^{4/3}/(4\alpha)^{4/3}\\
&=
\frac{|y|^{4/3}}{4^{1/3}\alpha^{1/3}}-
\frac{|y|^{4/3}}{4^{4/3}\alpha^{1/3}}
=\frac{|y|^{4/3}}{4^{1/3}\alpha^{1/3}}\big(1-\tfrac{1}{4} \big)
=\frac{3|y|^{4/3}}{4(4\alpha)^{1/3}}
\end{align*}
as claimed. 

To compute $\prox_h(y)$, we utilize \cref{p:proxquartic}.
Obtain fresh values of $p,q,\Delta$, we have this time
$p=1/(4\alpha)>0$ and $q=-y/(4\alpha)$ (see \cref{e:proxquarticpq}).
Hence $\Delta = (p/3)^3+(q/2)^2 = (1+27\alpha y^2)/(1728\alpha^3)$ and so
$\sqrt{\Delta}=\sqrt{1+27\alpha y^2}/(8(3\alpha)^{3/2})$. 
Now $-\beta/(4\alpha)=0$ and $-q/2=y/(8\alpha)$, so \cref{e:proxquartic} yields 
\begin{align*}
\prox_h(y) &= 
\sqrt[\mathlarger 3]{\frac{y}{8\alpha}+ \frac{\sqrt{1+27\alpha y^2}}{8(3\alpha)^{3/2}}}
+
\sqrt[\mathlarger 3]{\frac{y}{8\alpha}- \frac{\sqrt{1+27\alpha y^2}}{8(3\alpha)^{3/2}}}\\
&= 
\frac{1}{2}\sqrt[\mathlarger 3]{\frac{y}{\alpha}+ \sqrt{\frac{1+27\alpha y^2}{27\alpha^3}}}
+
\frac{1}{2}\sqrt[\mathlarger 3]{\frac{y}{\alpha}- \sqrt{\frac{1+27\alpha y^2}{27\alpha^3}}}
\end{align*}
as claimed. 
\end{proof}

\begin{remark}
The Fenchel conjugate formula \cref{e:alpha4*} is known and can also be computed by 
combining, e.g., \cite[Example~13.2(i) with Proposition~13.23(i)]{BC2017}. 
The prox formula \cref{e:221215a} appears --- with a typo though --- in \cite[Example~24.38(v)]{BC2017}. 
Finally, for a Maple implementation for quartics, see \cite{Yves}. 
\end{remark}

\section{The proximal mapping of $\alpha/x$}

\label{sec:alpha/x}

In this section, we study the convex reciprocal function 
\begin{empheq}[box=\mybluebox]{equation}
\label{e:alpha/x}
h(x) := \begin{cases} \alpha/x, &\text{if $x>0$;}\\
\pinf, &\text{if $x\leq 0$,}
\end{cases}\qquad
\text{where $\alpha> 0$.}
\end{empheq}
The Fenchel conjugate $h^*$, which requires only solving a \emph{quadratic} equation, 
is essentially known (e.g., combine \cite[Example~13.2(ii) and Proposition~13.23(i)]{BC2017}), and given by 
\begin{equation}
h^*(y) = \begin{cases}
-2\sqrt{-\alpha y}, &\text{if $y\leq 0$;}\\
\pinf, &\text{if $y>0$.}
\end{cases}
\end{equation}

The purpose of this section is to explicitly compute $\prox_h$.
We have the following result:

\begin{proposition}
\label{p:alpha/x}
Suppose that $h$ is given by \cref{e:alpha/x},
and let $y\in\RR$. 
Set $y_0 := -3\sqrt[3]{\alpha/4}\approx -1.88988 \sqrt[3]{\alpha}$. 
Then we have the following three possibilities:
\begin{enumerate}
\item If $y_0<y$, then 
\begin{align*}
\prox_h(y)
&= \frac{y}{3} + 
\sqrt[\mathlarger 3]{\frac{\alpha}{2}+\Big(\frac{y}{3}\Big)^3 +\sqrt{\alpha \Big( \frac{\alpha}{4}+\Big(\frac{y}{3}\Big)^3 \Big)}}
+\sqrt[\mathlarger 3]{\frac{\alpha}{2}+\Big(\frac{y}{3}\Big)^3 -\sqrt{\alpha \Big( \frac{\alpha}{4}+\Big(\frac{y}{3}\Big)^3 \Big)}}. 
\end{align*}
\item If $y=y_0$, then 
\begin{equation*}
\prox_h(y_0) = \sqrt[3]{\alpha}/\sqrt[3]{4}\approx 0.62996\sqrt[3]{\alpha}.
\end{equation*}
\item If $y<y_0$, then 
\begin{equation*}
\prox_h(y)= 
\frac{y}{3}\bigg(1- 2\cos\Big(\frac{1}{3}\arccos\frac{(y/3)^3+\alpha/2}{-(y/3)^{3}}\Big)\bigg). 
\end{equation*}
\end{enumerate}
\end{proposition}
\begin{proof}
Because $\dom h = \RPP$, we must find the \emph{positive}
solution of the equation $h'(x)+x-y=0$.
Since $h'(x)=-\alpha x^{-2}$, we are looking for the (necessarily unique) positive solution
of $x^2(h'(x)+x-y)=0$, i.e., of 
\begin{equation*}
x^3-yx^2-\alpha=0.
\end{equation*}
This fits the pattern of \cref{e:gencubic} in \cref{sec:gencubic},
with parameters
$a=1$, $b=-y$, $c=0$, and $d=-\alpha$.
As in \cref{e:defpq}, we set 
We have
\begin{align}
\label{e:Ilop}
p &:=
\frac{3ac-b^2}{3a^2}
= -\frac{y^2}{3} 
\begin{cases}
<0, &\text{if $y\neq 0$;}\\
=0, &\text{if $y=0$}
\end{cases}
\end{align}
and 
\begin{align*}
q &:=
\frac{27a^2d+2b^3-9abc}{27a^3}
= -\alpha - 2(y/3)^3. 
\end{align*}
Next, 
\begin{align*}
\Delta 
&= (p/3)^3 + (q/2)^2
= -y^6/9^3 + (\alpha+2(y/3)^3)^2/4\\
&=
-(y/3)^6+\alpha^2/4+\alpha(y/3)^3+(y/3)^6\\
&= \alpha\big(\alpha/4 + (y/3)^3\big).
\end{align*}
Hence 
\begin{equation}
\Delta
\begin{cases}
< 0 \Leftrightarrow y < y_0; \\
= 0 \Leftrightarrow y = y_0; \\
> 0 \Leftrightarrow y>y_0,
\end{cases}
\quad \text{where}\;\; y_0 := -\frac{3}{\sqrt[3]{4}}\sqrt[3]{\alpha}\approx -1.88988 \sqrt[3]{\alpha}. 
\end{equation}
Now set 
\begin{equation}
\label{e:Ilofirst}
x_0 := -\frac{b}{3a} = \frac{y}{3}. 
\end{equation}
Note that 
\begin{equation}
\label{e:Ilo-q/2} 
-q/2=(y/3)^3+\alpha/2.
\end{equation}
We now discuss the three possibilities from \cref{c:genroots} --- these will correspond to the three items of the result!

\emph{Case~1}: $b^2=3ac$ or $\Delta>0$; equivalently, $y=0$ or $y>y_0$; equivalently, $y_0<y$.\\
Then \cref{c:genroots}\cref{c:genroots1} yields
\begin{align*}
\prox_h(y)
&= x_0 + \sqrt[\mathlarger 3]{\frac{-q}{2}+\sqrt{\Delta}} + \sqrt[\mathlarger 3]{\frac{-q}{2}-\sqrt{\Delta}}\\
&= \frac{y}{3} + 
\sqrt[\mathlarger 3]{\frac{\alpha}{2}+\Big(\frac{y}{3}\Big)^3 +\sqrt{\alpha \Big( \frac{\alpha}{4}+\Big(\frac{y}{3}\Big)^3 \Big)}}
+\sqrt[\mathlarger 3]{\frac{\alpha}{2}+\Big(\frac{y}{3}\Big)^3 -\sqrt{\alpha \Big( \frac{\alpha}{4}+\Big(\frac{y}{3}\Big)^3 \Big)}}
\end{align*}
as claimed.

\emph{Case~2}: $\Delta=0$; equivalently, $y=y_0$.\\
Then \cref{c:genroots}\cref{c:genroots2} yields two distinct real roots.
We can take a short cut here, though:
By exploiting the continuity of $\prox_h$ at $y_0$ via 
$\prox_h(y_0)=\lim_{y\to y_0^+}\prox_h(y)$, we get
\begin{equation*}
\prox_h(y_0) = 
\frac{y_0}{3}+2\sqrt[\mathlarger 3]{\frac{\alpha}{2}+\Big(\frac{y_0}{3}\Big)^3 }
= \sqrt[3]{\alpha}/\sqrt[3]{4}\approx 0.62996\sqrt[3]{\alpha}. 
\end{equation*}

\emph{Case~3}: $\Delta<0$; equivalently, $y<y_0$.\\
By uniqueness of $\prox_h(y)$, the desired root must be the \emph{largest} (and the only positive) real root offered in this case (see \cref{c:genroots}\cref{c:genroots3}): 
\begin{equation}
\prox_h(y)= 
x_0 + 2(-p/3)^{1/2}\cos\Big(\frac{\theta}{3}\Big), 
\quad\text{where}\;\;
\theta := \arccos\frac{-q/2}{(-p/3)^{3/2}}, 
\end{equation}
By \cref{e:Ilop},
$-p/3=y^2/9$; thus, using $y<y_0<0$, we have $(-p/3)^{1/2}=-y/3$, 
$(-p/3)^{3/2}=-(y/3)^3$, and \cref{e:Ilo-q/2} yields
$-(q/2)/(-p/3)^{3/2}=-((y/3)^3+\alpha/2)/(y/3)^3$. 
This and \cref{e:Ilofirst} results in 
\begin{equation}
\prox_h(y)= 
\frac{y}{3} - 2\frac{y}{3}\cos\Big(\frac{\theta}{3}\Big), 
\quad\text{where}\;\;
\theta := \arccos \bigg(-\frac{(y/3)^3+\alpha/2}{(y/3)^3}\bigg).
\end{equation}
\end{proof}

\begin{remark}
Suppose that $\alpha=1$. Then $\prox_h$ was discussed in \cite{p-o.net};
however, no explicit formulae were presented. 
For a visualization of $\prox_h$ in this case, see \cref{fig2}. 

\begin{figure}
\centering
   \includegraphics[width=0.5\textwidth]{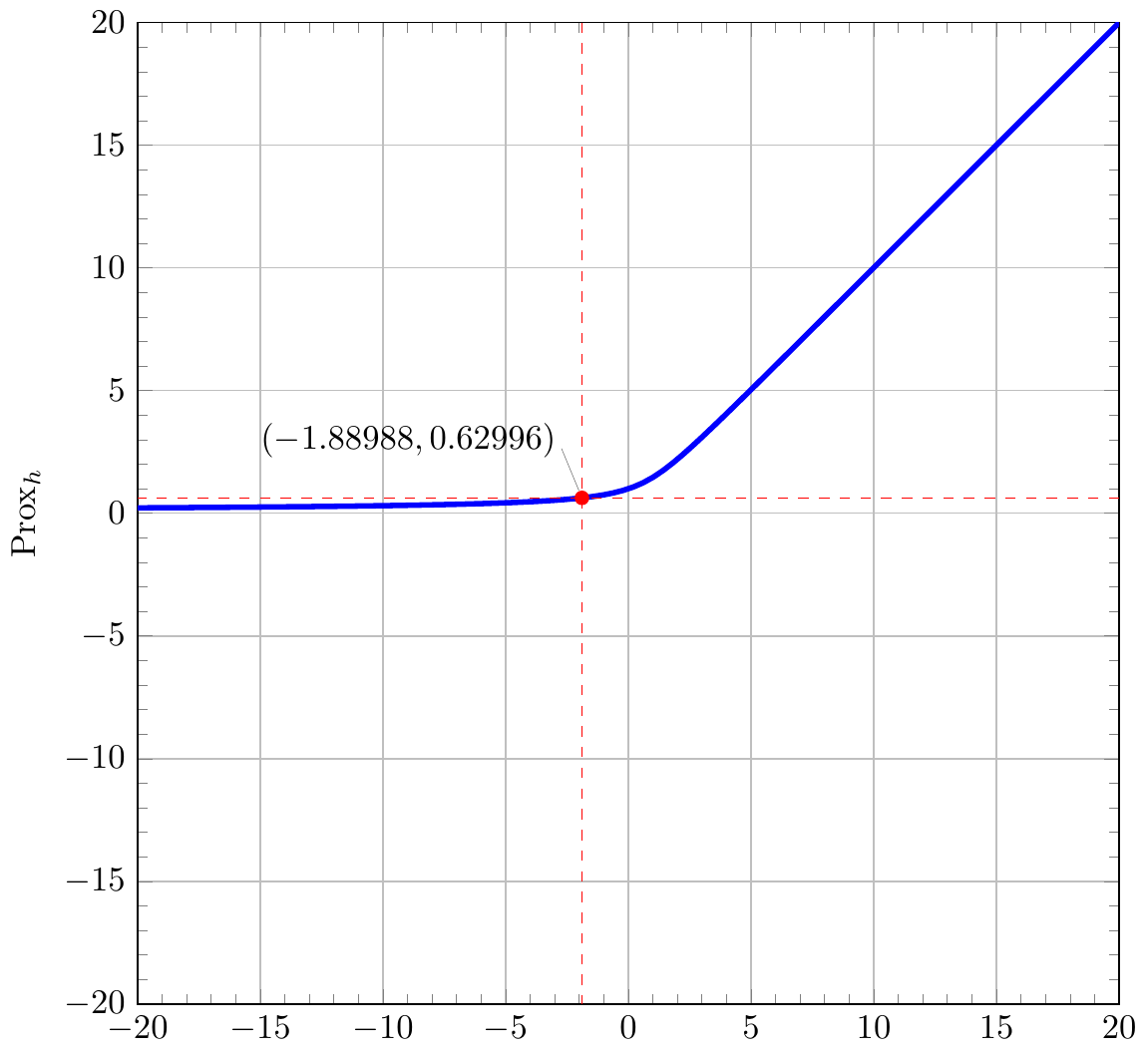}
\caption{A visualization of \cref{p:alpha/x} when $\alpha=1$.
}
\label{fig2}
\end{figure}
\end{remark}

\section{Projection onto epigraph of a parabola}

\label{sec:projepipar}

In this section, we study projection onto the epigraph of the function
\begin{empheq}[box=\mybluebox]{equation}
h\colon \RR^n\to\RR\colon \bx \mapsto \alpha\|\bx\|^2, 
\quad 
\text{where $\alpha> 0$.}
\end{empheq}

\begin{theorem}
\label{t:projepipar}
Set $E := \epi h \subseteq \RR^{n+1}$.
Let $(\by,\eta)\in(\RR^n\times\RR)$. 
If $(\by,\eta)\in E$, then $P_E(\by,\eta)=(\by,\eta)$.
So we assume that 
$(\by,\eta)\in(\RR^n\times\RR)\smallsetminus E$, i.e., 
$\alpha\|\by\|^2>\eta$. Set $\nu := \|\by\|\geq 0$, 
\begin{equation}
p := -\frac{(2\alpha \eta-1)^2}{12\alpha^2}, 
\quad
q := \frac{(2\alpha \eta-1)^3-27\alpha^2\nu^2}{108\alpha^3},
\end{equation}
$\Delta := (p/3)^3+(q/2)^2 = (27\alpha^2{\nu^2}-2(2\alpha\eta-1)^3)\nu^2/(1728\alpha^4)$, and 
\begin{equation}
x := \begin{cases}
 -\dfrac{\alpha \eta+1}{3\alpha} + 
\sqrt[\mathlarger 3]{-q/2 +\sqrt{\Delta}}
+
\sqrt[\mathlarger 3]{-q/2 -\sqrt{\Delta}}, &\text{if $\Delta\geq 0$;}\\[+5mm]
 -\dfrac{\alpha \eta+1}{3\alpha} +
\dfrac{|2\alpha \eta-1|}{3\alpha}\cos\Big(\dfrac{1}{3}\arccos \dfrac{-q/2}{(-p/3)^{3/2}}\Big), &\text{if $\Delta<0$.}
\end{cases}
\end{equation}
Then
\begin{equation}
P_E(\by,\eta) = \Big(\frac{\by}{1+2\alpha x},\eta+x\Big).
\end{equation}
See \cref{fig3} for an illustration for the case $\alpha=1/2$. 
\end{theorem}
\begin{proof}
For $x\geq 0$, we have $x h=x\alpha\|\cdot\|^2$,
$x \nabla h=2\alpha x\Id$,
$\Id+ x \nabla h=(1+2\alpha x)\Id$ and therefore
$\prox_{x h}=(1+2\alpha x)^{-1}\Id$. 
In view of \cite[Theorem~6.36]{Beck2}, we must first find a positive root $x$ of 
$\varphi(x) := h(\prox_{x h}(\by))-x-\eta=\alpha \|\by\|^2/(1+2\alpha x)^2-x-\eta=0$. 
Note that $\varphi(0)>0$, that $\varphi$ is strictly decreasing on $\RP$, and that 
$\varphi(x)\to\minf$ as $x\to\pinf$. 
Hence $\varphi$ has \emph{exactly one} positive root. 
Multiplying by $(1+2\alpha x)^2>0$, where $x>0$, results in the cubic
$\alpha \nu^2-(x+\eta)(1+2\alpha x)^2=0$, 
which must have \emph{exactly one} positive root. 
Re-arranging leads us to 
\begin{equation}
f(x) := 4\alpha^2 x^3 + 4\alpha(\alpha \eta+1)x^2 + (4\alpha \eta+1)x +\eta-\alpha \nu^2=0,
\end{equation}
a cubic which we know has exactly one positive root. 
As in \cref{sec:gencubic}, we set
\begin{equation}
a := 4\alpha^2,
\;\;
b := 4\alpha(\alpha \eta+1),
\;\;
c := 4\alpha \eta+1,
\;\;
d := \eta-\alpha \nu^2 < 0,
\end{equation}
\begin{equation}
p := \frac{3ac-b^2}{3a^2} = -\frac{(2\alpha \eta-1)^2}{12\alpha^2}\leq 0, 
\end{equation}
and 
\begin{equation}
q := \frac{27a^2d+2b^3-9abc}{27a^3} = \frac{8\alpha^3\eta^3-12\alpha^2\eta^2+6\alpha \eta-27\alpha^2\nu^2-1}{108\alpha^3}.
\end{equation}
We then have
\begin{subequations}
\begin{align}
\Delta &:= (p/3)^3+(q/2)^2 = \frac{\big(8\alpha^3\eta^3-12\alpha^2\eta^2+6\alpha \eta-27\alpha^2\nu^2-1\big)^2-(2\alpha \eta-1)^6}{(6\alpha)^6}\\
&=-\frac{\nu^2}{1728\alpha^4}\big(16\alpha^3\eta^3-24\alpha^2\eta^2+12\alpha \eta-27\alpha^2\nu^2-2 \big)\\
&=\frac{\nu^2}{1728\alpha^4}\big(27\alpha^2\nu^2-2(2\alpha\eta-1)^3\big), 
\end{align}
\end{subequations}
as claimed. 
Utilizing \cref{c:genroots}, we have
\begin{subequations}
\begin{equation}
x = -\frac{\alpha \eta+1}{3\alpha} + 
\sqrt[\mathlarger 3]{-q/2 +\sqrt{\Delta}}
+
\sqrt[\mathlarger 3]{-q/2 -\sqrt{\Delta}},
\quad\text{if $\Delta\geq 0$}
\end{equation}
and
\begin{equation}
x = -\frac{\alpha \eta+1}{3\alpha} +
\frac{|2\alpha \eta-1|}{3\alpha}\cos\Big(\frac{1}{3}\arccos \frac{-q/2}{(-p/3)^{3/2}}\Big)
\quad\text{if $\Delta< 0$.}
\end{equation}
\end{subequations}
(Because we know there is \emph{exactly one} positive root, it is clear that we must pick 
$r_0$ in \cref{c:genroots}\cref{c:genroots3} when $\Delta<0$.)
Finally, \cite[Theorem~6.36]{Beck2} yields
\begin{equation}
P_E(\by,\eta) = \big(\prox_{xh}(\by),\eta+x\big)
= \Big(\frac{\by}{1+2\alpha x},\eta+x\Big)
\end{equation}
as claimed. 
\end{proof}

\begin{figure}
\centering
\includegraphics[width=0.75\textwidth]{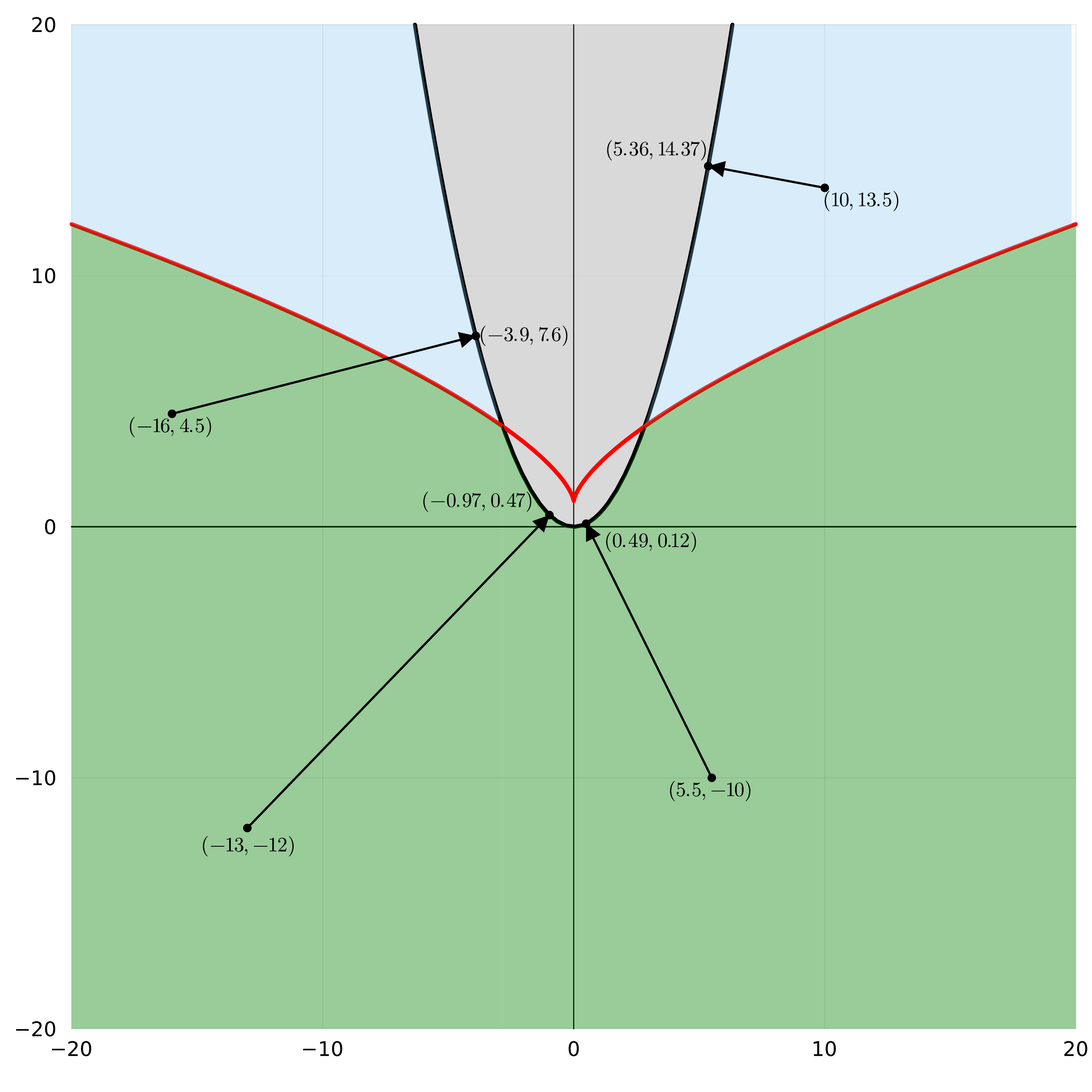}
\caption{A visualization of \cref{t:projepipar} when $n=1$ and $\alpha=1/2$. 
The epigraph is shown in gray. 
The red curve corresponds to $\Delta=0$, the green region to $\Delta<0$ (where trig functions are used) and the blue region to $\Delta>0$.}
\label{fig3}
\end{figure}

\section{On the projection of a rectangular hyperbolic paraboloid }
\label{sec:missing}

In this section, $X$ is a real Hilbert space and we set 
\begin{empheq}[box=\mybluebox]{equation}
S := \menge{(\bx,\by,\gamma)\in X\times X\times\RR}{\scal{\bx}{\by}=\alpha\gamma}, 
\quad 
\text{where $\alpha\in\RR\smallsetminus\{0\}$.}
\end{empheq}
Using the Hilbert product space norm
$\|(\bx,\by,\gamma)\| := \sqrt{\|\bx\|^2+\|\by\|^2+\beta^2\gamma^2}$, 
where $\beta>0$, we are interested 
in finding the projection onto $S$.
Various cases were discussed in \cite{hypar}, 
but 3 were treated only implicitly.
Armed with the cubic, we are now able to treat two of these cases explicitly (the remaining case features a quintic and remains hard).
The first case concerns
\begin{equation}
\label{e:case1bla}
P_S(\bz,-\bz,\gamma),
\quad\text{when}\;\;
\bz\in X\smallsetminus\{0\}\;\text{and}\;
\alpha(\gamma-\alpha/\beta^2)<-\|\bz\|^2/4,
\end{equation}
while the second case is 
\begin{equation}
\label{e:case2bla}
P_S(\bz,\bz,\gamma),
\quad\text{when}\;\;
\bz\in X\smallsetminus\{0\}\;\text{and}\;
\alpha(\gamma+\alpha/\beta^2)>\|\bz\|^2/4.
\end{equation}

\subsection{The case when \cref{e:case1bla} holds} 

\begin{theorem}
Suppose 
$\bz\in X\smallsetminus\{0\}$, set
\begin{equation}
\zeta := \|\bz\|>0,
\end{equation}
and assume that 
\begin{equation}
\label{e:230109a}
\alpha(\gamma-\alpha/\beta^2)<-\zeta^2/4. 
\end{equation}
Set 
\begin{equation}
\label{e:230109b}
p := -\frac{(\alpha+\beta^{2} \gamma )^{2}}{3  \alpha^{2}},
\quad
q := 
\frac{2(\alpha+\beta^2\gamma)^3}{27\alpha^3}
+ \frac{\beta^2\zeta^2}{\alpha^2},
\end{equation}
and 
\begin{equation}
\label{e:230109c}
\Delta := (p/3)^3+(q/2)^2= 
\frac{\beta^2\zeta^2}{\alpha^2}\Big(\frac{\beta^2\zeta^2}{4\alpha^2}+\frac{(\alpha+\beta^2\gamma)^3} {27\alpha^3}\Big). 
\end{equation}
If $\Delta\geq 0$, then set
\begin{equation}
\label{e:xDnonneg}
x := \frac{2\alpha-\beta^2\gamma}{3\alpha}+ \sqrt[\mathlarger 3]{\frac{-q}{2}+ \sqrt{\Delta}}
+
\sqrt[\mathlarger 3]{\frac{-q}{2}-\sqrt{\Delta}}; 
\end{equation}
and if $\Delta<0$, then set
\begin{subequations}
\label{e:xDneg}
\begin{equation}
x := \frac{2\alpha-\beta^2\gamma}{3\alpha}
+\delta\frac{2(\alpha+\beta^2\gamma)}{3\alpha}\cos\bigg(\frac{1}{3}\Big((3+\delta)\pi+\arccos \frac{-q/2}{(-p/3)^{3/2}}\Big) \bigg)
\end{equation}
where 
\begin{equation}
\delta := \sign(\alpha^2+\alpha\beta^2\gamma)\in\{-1,0,1\}.
\end{equation}
\end{subequations}
Then $-1<x<1$ and 
\begin{equation}
P_S(\bz,-\bz,\gamma) = \Big(\frac{\bz}{1-x},\frac{-\bz}{1-x},\gamma+\frac{\alpha x}{\beta^2} \Big).
\end{equation}
\end{theorem}
\begin{proof}
By \cite[Theorem~4.1(ii)(a)]{hypar}, 
there exists a \emph{unique} $x\in\left]-1,1\right[$ such that 
\begin{equation}
\frac{2\zeta^2}{(1-x)^2}+\frac{2\alpha^2x}{\beta^2}+2\alpha\gamma=0; 
\end{equation}
multiplying by 
$\beta^2(1-x)^2/2>0$ yields the cubic
\begin{equation}
f(x) := ax^3+bx^2+cx+d=0
\end{equation}
where 
\begin{equation}
a := 
\alpha^{2}>0,
\;\;
b := 
\alpha \beta^{2} \gamma - 2\alpha^{2},
\;\;
c := 
\alpha^2-2\alpha \beta^{2} \gamma, 
\;\;
d := 
\alpha \beta^{2} \gamma + \beta^{2} \zeta^{2}. 
\end{equation}
Our strategy is to systematically discuss all cases of 
\cref{t:genroots} and then combine cases as much as possible.
As usual, we set 
\begin{equation}
\label{e:230109d}
x_0 := -\frac{b}{3a} = 
\frac{2\alpha-\beta^{2}\gamma }{3 \alpha}
= 
\frac{2\alpha^2-\alpha \beta^{2}\gamma }{3 \alpha^{2}} 
\quad\text{and}\quad
p := \frac{3ac-b^2}{3a^2} = 
-\frac{(\alpha+\beta^{2} \gamma )^{2}}{3  \alpha^{2}} \leq 0,
\end{equation}
and we note that the definition of $p$ is consistent with the one given
in \cref{e:230109b}.
We have the characterization 
\begin{equation}
p = 0 
\;\Leftrightarrow\;
\alpha+\beta^2\gamma=0
\;\Leftrightarrow\;
\gamma=-\alpha/\beta^2. 
\end{equation}
Again as usual, we set 
\begin{equation}
q := \frac{27a^2d+2b^2-9abc}{27a^3} = 
\frac{2(\alpha+\beta^2\gamma)^3}{27\alpha^3}
+ \frac{\beta^2\zeta^2}{\alpha^2},
\end{equation}
which matches \cref{e:230109b}, and of course
\begin{equation}
\Delta := (p/3)^3+(q/2)^2= 
\frac{\beta^2\zeta^2}{\alpha^2}\Big(\frac{\beta^2\zeta^2}{4\alpha^2}+\frac{(\alpha+\beta^2\gamma)^3} {27\alpha^3}\Big),
\end{equation}
which matches \cref{e:230109c}.
We now systematically discuss the case of \cref{t:genroots}.

\noindent
\emph{Case~1:} $p<0$, i.e., $\alpha+\beta^2\gamma\neq 0$ by \cref{e:230109b}.

\emph{Case~1(a):} $p<0$ and $\Delta>0$.\\
Then \cref{t:genroots}\cref{t:genroots1a} and 
the definition $x_0$ in \cref{e:230109d} yield
\cref{e:xDnonneg}. 

\emph{Case~1(b):} $p<0$ and $\Delta=0$.\\
By \cref{t:genroots}\cref{t:genroots1b}, there are two roots,
$x_0+3q/p$ and $x_0-3q/(2p)$, one of which lies in $\left]-1,1\right[$.
Now 
\begin{align*}
x_0-\frac{3q}{2p}-1
&= \frac{2\alpha-\beta^2\gamma}{3\alpha}
-\frac{3}{2}\frac{2(\alpha+\beta^2\gamma)^3+27\alpha\beta^2\zeta^2} {27\alpha^3} 
\frac{-3\alpha^2}{(\alpha+\beta^2\gamma)^2}-1\\
&=
\frac{2\alpha-\beta^2\gamma}{3\alpha}+
\frac{(\alpha+\beta^2\gamma)^3+27\alpha\beta^2\zeta^2/2}{3\alpha(\alpha+\beta^2\gamma)^2}-1\\
&=
\frac{\big((2\alpha-\beta^2\gamma)+(\alpha+\beta^2\gamma)-(3\alpha)\big)}{3\alpha(\alpha+\beta^2\gamma)^2}(\alpha+\beta^2\gamma)^2
+ \frac{27\alpha\beta^2\zeta^2/2}{3\alpha(\alpha+\beta^2\gamma)^2}\\
&= \frac{9\beta^2\zeta^2}{2(\alpha+\beta^2\gamma)^2}\\
&\geq 0; 
\end{align*}
hence the root $x_0-3q/(2p)$ lies in $\left[1,\pinf\right[$ and therefore
our desired root is the remaining one, namely $x_0+3q/p$, which also
allows us to use the representation \cref{e:xDnonneg}.

\emph{Case~1(c):} $p<0$ and $\Delta<0$.\\
According to \cref{t:genroots}\cref{t:genroots1c}, we have three distinct real roots, but there is information about their location. We must locate the root in $\left]-1,1\right[$. 
First, 
$b^2-3ac=(\alpha^2+\alpha\beta^2\gamma)^2$
which yields
$\sqrt{b^2-3ac}=|\alpha^2+\alpha\beta^2\gamma|$. 
This and the definition of $b$ yields
\begin{align*}
x_\pm &:= \frac{-b\pm\sqrt{b^2-3ac}}{3a}
= \frac{2\alpha^2-\alpha\beta^2\gamma\pm |\alpha^2+\alpha\beta^2\gamma|}{3\alpha^2}\\
&=\frac{1}{3\alpha^2}
\frac{(3\alpha^2)+(\alpha^2-2\alpha\beta^2\gamma)\pm\big|(3\alpha^2)-(\alpha^2-2\alpha\beta^2\gamma)\big|}{2}.
\end{align*}
Hence
\begin{equation}
x_-=\frac{\min\{3\alpha^2,\alpha^2-2\alpha\beta^2\gamma\}}{3\alpha^2}
<\frac{\max\{3\alpha^2,\alpha^2-2\alpha\beta^2\gamma\}}{3\alpha^2}=x_+.
\end{equation}
We now bifurcate one last time.

\emph{Case~1(c)($+$):} $p<0$, $\Delta<0$, and $\alpha^2+\alpha\beta^2\gamma>0$.\\
Then $3\alpha^2>\alpha^2-2\alpha\beta^2\gamma$ and
therefore $x_+=1$. 
It follows that our desired root $x$ is the ``middle root'' corresponding to 
$k=2$ in \cref{t:genroots}\cref{t:genroots1c}:
\begin{align*}
x
&=x_0+2(-p/3)^{1/2}\cos\bigg(\frac{1}{3}\Big(4\pi+\arccos \frac{-q/2}{(-p/3)^{3/2}}\Big) \bigg)\\
&=
\frac{2\alpha-\beta^2\gamma}{3\alpha}
+\frac{2|\alpha+\beta^2\gamma|}{3|\alpha|}\cos\bigg(\frac{1}{3}\Big(4\pi+\arccos \frac{-q/2}{(-p/3)^{3/2}}\Big) \bigg)\\
&=
\frac{2\alpha-\beta^2\gamma}{3\alpha}
+\frac{2(\alpha+\beta^2\gamma)}{3\alpha}\cos\bigg(\frac{1}{3}\Big(4\pi+\arccos \frac{-q/2}{(-p/3)^{3/2}}\Big) \bigg),
\end{align*}
where in the last line we used the assumption to deduce that 
$|\alpha+\beta^2\gamma|/|\alpha|
=|\alpha^2+\alpha\beta^2\gamma|/\alpha^2
=(\alpha^2+\alpha\beta^2\gamma)/\alpha^2=
(\alpha+\beta^2\gamma)/\alpha$. 

\emph{Case~1(c)($-$):} $p<0$, $\Delta<0$, and $\alpha^2+\alpha\beta^2\gamma\leq 0$.\\
Then 
$3\alpha^2\leq \alpha^2-2\alpha\beta^2\gamma$ 
and therefore $x_-=1$. 
It follows that our desired root is the ``smallest root'' corresponding to 
$k=1$ in \cref{t:genroots}\cref{t:genroots1c}:
\begin{align*}
x
&=x_0+2(-p/3)^{1/2}\cos\bigg(\frac{1}{3}\Big(2\pi+\arccos \frac{-q/2}{(-p/3)^{3/2}}\Big) \bigg)\\
&=
\frac{2\alpha-\beta^2\gamma}{3\alpha}
+\frac{2|\alpha+\beta^2\gamma|}{3|\alpha|}\cos\bigg(\frac{1}{3}\Big(2\pi+\arccos \frac{-q/2}{(-p/3)^{3/2}}\Big) \bigg)\\
&=
\frac{2\alpha-\beta^2\gamma}{3\alpha}
-\frac{2(\alpha+\beta^2\gamma)}{3\alpha}\cos\bigg(\frac{1}{3}\Big(2\pi+\arccos \frac{-q/2}{(-p/3)^{3/2}}\Big) \bigg),
\end{align*}
where in the last line we used the assumption to deduce that 
$|\alpha+\beta^2\gamma|/|\alpha|
=|\alpha^2+\alpha\beta^2\gamma|/\alpha^2
=-(\alpha^2+\alpha\beta^2\gamma)/\alpha^2=
-(\alpha+\beta^2\gamma)/\alpha$. 

Note that the last two cases can be combined to obtain \cref{e:xDneg}.

\noindent
\emph{Case~2:} $p=0$, i.e., $\alpha+\beta^2\gamma= 0$ by \cref{e:230109b}.
Then $\Delta = (q/2)^2\geq 0$; hence, $\sqrt{\Delta}=|q|/2$ and 
thus $\{-q/2\pm\sqrt{\Delta}\} = \{-q,0\}$.
By \cref{t:genroots}\cref{t:genroots2}, the only real root is 
$x_0+(-q)^{1/3}=x_0+(-q/2+\sqrt{\Delta})^{1/3} + (-q/2-\sqrt{\Delta})^{1/3}$ which is the same as \cref{e:xDnonneg} using \cref{e:230109d}.

\noindent
\emph{Case~3:} $p>0$. In vie of \cref{e:230109b}, this case never occurs.
\end{proof}

\subsection{The case when \cref{e:case2bla} holds} 
\begin{theorem}
Suppose 
$\bz\in X\smallsetminus\{0\}$, set
\begin{equation}
\zeta := \|\bz\|>0,
\end{equation}
and assume that 
\begin{equation}
\label{e:230111a}
\alpha(\gamma+\alpha/\beta^2)>\zeta^2/4. 
\end{equation}
Set 
\begin{equation}
\label{e:230111b}
p := -\frac{(\beta^{2} \gamma - \alpha )^{2}}{3  \alpha^{2}},
\quad
q := 
\frac{2(\beta^2\gamma - \alpha)^3}{27\alpha^3}
- \frac{\beta^2\zeta^2}{\alpha^2},
\end{equation}
and 
\begin{equation}
\label{e:230111c}
\Delta := (p/3)^3+(q/2)^2= 
\frac{\beta^2\zeta^2}{\alpha^2}\Big(\frac{\beta^2\zeta^2}{4\alpha^2}-\frac{(\beta^2\gamma - \alpha)^3} {27\alpha^3}\Big). 
\end{equation}
If $\Delta\geq 0$, then set
\begin{equation}
\label{e:xDnonneg11}
x := -\frac{2\alpha+\beta^2\gamma}{3\alpha}+ \sqrt[\mathlarger 3]{\frac{-q}{2}+ \sqrt{\Delta}}
+
\sqrt[\mathlarger 3]{\frac{-q}{2}-\sqrt{\Delta}}; 
\end{equation}
and if $\Delta<0$, then set
\begin{subequations}
\label{e:xDneg11}
\begin{equation}
x := -\frac{2\alpha+\beta^2\gamma}{3\alpha}
+\delta\frac{2(\alpha-\beta^2\gamma)}{3\alpha}\cos\bigg(\frac{1}{3}\Big((2+2\delta)\pi+\arccos \frac{-q/2}{(-p/3)^{3/2}}\Big) \bigg)
\end{equation}
where 
\begin{equation}
\delta := \sign(\alpha^2-\alpha\beta^2\gamma)\in\{-1,0,1\}.
\end{equation}
\end{subequations}
Then $-1<x<1$ and 
\begin{equation}
P_S(\bz,\bz,\gamma) = \Big(\frac{\bz}{1+x},\frac{\bz}{1+x},\gamma+\frac{\alpha x}{\beta^2} \Big).
\end{equation}
\end{theorem}
\begin{proof}
By \cite[Theorem~4.1(iii)(a)]{hypar}, 
there exists a \emph{unique} $x\in\left]-1,1\right[$ such that 
\begin{equation}
\frac{2\zeta^2}{(1+x)^2}-\frac{2\alpha^2x}{\beta^2}-2\alpha\gamma=0; 
\end{equation}
multiplying by 
$-\beta^2(1+x)^2/2<0$ yields the cubic
\begin{equation}
f(x) := ax^3+bx^2+cx+d=0
\end{equation}
where 
\begin{equation}
a := 
\alpha^{2}>0,
\;\;
b := 
\alpha \beta^{2} \gamma + 2\alpha^{2},
\;\;
c := 
\alpha^2+2\alpha \beta^{2} \gamma, 
\;\;
d := 
\alpha \beta^{2} \gamma - 
\beta^{2} \zeta^{2}. 
\end{equation}
Our strategy is to systematically discuss all cases of 
\cref{t:genroots} and then combine cases as much as possible.
As usual, we set 
\begin{equation}
\label{e:230111d}
x_0 := -\frac{b}{3a} = 
-\frac{2\alpha+\beta^{2}\gamma }{3 \alpha}
= 
-\frac{2\alpha^2+\alpha \beta^{2}\gamma }{3 \alpha^{2}} 
\quad\text{and}\quad
p := \frac{3ac-b^2}{3a^2} = 
-\frac{(\beta^{2} \gamma-\alpha)^{2}}{3  \alpha^{2}} \leq 0,
\end{equation}
and we note that the definition of $p$ is consistent with the one given
in \cref{e:230111b}.
We have the characterization 
\begin{equation}
p = 0 
\;\Leftrightarrow\;
\beta^2\gamma-\alpha=0
\;\Leftrightarrow\;
\gamma=\alpha/\beta^2. 
\end{equation}
Again as usual, we set 
\begin{equation}
q := \frac{27a^2d+2b^2-9abc}{27a^3} = 
\frac{2(\beta^2\gamma-\alpha)^3}{27\alpha^3}
- \frac{\beta^2\zeta^2}{\alpha^2},
\end{equation}
which matches \cref{e:230111b}, and of course
\begin{equation}
\Delta := (p/3)^3+(q/2)^2= 
\frac{\beta^2\zeta^2}{\alpha^2}\Big(\frac{\beta^2\zeta^2}{4\alpha^2}-\frac{(\beta^2\gamma-\alpha)^3} {27\alpha^3}\Big),
\end{equation}
which matches \cref{e:230111c}.
We now systematically discuss the case of \cref{t:genroots}.

\noindent
\emph{Case~1:} $p<0$, i.e., $\beta^2\gamma-\alpha\neq 0$ by \cref{e:230111b}.

\emph{Case~1(a):} $p<0$ and $\Delta>0$.\\
Then \cref{t:genroots}\cref{t:genroots1a} and 
the definition of $x_0$ in \cref{e:230111d} yield
\cref{e:xDnonneg11}. 

\emph{Case~1(b):} $p<0$ and $\Delta=0$.\\
By \cref{t:genroots}\cref{t:genroots1b}, there are two roots,
$x_0+3q/p$ and $x_0-3q/(2p)$, one of which lies in $\left]-1,1\right[$.
Now 
\begin{align*}
x_0-\frac{3q}{2p}+1
&= -\frac{2\alpha+\beta^2\gamma}{3\alpha}
-\frac{3}{2}\frac{2(\beta^2\gamma-\alpha)^3-27\alpha\beta^2\zeta^2} {27\alpha^3} 
\frac{-3\alpha^2}{(\beta^2\gamma-\alpha)^2}+1\\
&=
-\frac{2\alpha+\beta^2\gamma}{3\alpha}+
\frac{(\beta^2\gamma-\alpha)^3-27\alpha\beta^2\zeta^2/2}{3\alpha(\beta^2\gamma-\alpha)^2}+1\\
&=
\frac{\big((-2\alpha-\beta^2\gamma)+(\beta^2\gamma-\alpha)+(3\alpha)\big)}{3\alpha(\beta^2\gamma-\alpha)^2}(\beta^2\gamma-\alpha)^2
-\frac{27\alpha\beta^2\zeta^2/2}{3\alpha(\beta^2\gamma-\alpha)^2}\\
&= -\frac{9\beta^2\zeta^2}{2(\beta^2\gamma-\alpha)^2}\\
&\leq 0; 
\end{align*}
hence the root $x_0-3q/(2p)$ lies in $\left]\minf,-1 \right]$ 
and therefore
our desired root is the remaining one, namely $x_0+3q/p$, which also
allows us to use the representation \cref{e:xDnonneg11}.

\emph{Case~1(c):} $p<0$ and $\Delta<0$.\\
According to \cref{t:genroots}\cref{t:genroots1c}, we have three distinct real roots, but there is information about their location. We must locate the root in $\left]-1,1\right[$. 
First, 
$b^2-3ac=(\alpha^2-\alpha\beta^2\gamma)^2$
which yields
$\sqrt{b^2-3ac}=|\alpha^2-\alpha\beta^2\gamma|$. 
This and the definition of $b$ yields
\begin{align*}
x_\pm &:= \frac{-b\pm\sqrt{b^2-3ac}}{3a}
= \frac{-2\alpha^2-\alpha\beta^2\gamma\pm |\alpha^2-\alpha\beta^2\gamma|}{3\alpha^2}\\
&=\frac{1}{3\alpha^2}
\frac{(-3\alpha^2)+(-\alpha^2-2\alpha\beta^2\gamma)\pm\big|(-3\alpha^2)-(-\alpha^2-2\alpha\beta^2\gamma)\big|}{2}.
\end{align*}
Hence
\begin{equation}
x_-=\frac{\min\{-3\alpha^2,-\alpha^2-2\alpha\beta^2\gamma\}}{3\alpha^2}
<\frac{\max\{-3\alpha^2,-\alpha^2-2\alpha\beta^2\gamma\}}{3\alpha^2}=x_+.
\end{equation}
We now bifurcate one last time.

\emph{Case~1(c)($+$):} $p<0$, $\Delta<0$, and $\alpha^2-\alpha\beta^2\gamma>0$.\\ 
Then $-3\alpha^2<-\alpha^2-2\alpha\beta^2\gamma$ and
therefore $x_-=-1$. 
It follows that our desired root $x$ is the ``middle root'' corresponding to 
$k=2$ in \cref{t:genroots}\cref{t:genroots1c}:
\begin{align*}
x
&=x_0+2(-p/3)^{1/2}\cos\bigg(\frac{1}{3}\Big(4\pi+\arccos \frac{-q/2}{(-p/3)^{3/2}}\Big) \bigg)\\
&=
-\frac{2\alpha+\beta^2\gamma}{3\alpha}
+\frac{2|\beta^2\gamma-\alpha|}{3|\alpha|}\cos\bigg(\frac{1}{3}\Big(4\pi+\arccos \frac{-q/2}{(-p/3)^{3/2}}\Big) \bigg)\\
&=
-\frac{2\alpha+\beta^2\gamma}{3\alpha}
+\frac{2(\alpha-\beta^2\gamma)}{3\alpha}\cos\bigg(\frac{1}{3}\Big(4\pi+\arccos \frac{-q/2}{(-p/3)^{3/2}}\Big) \bigg),
\end{align*}
where in the last line we used the assumption to deduce that 
$|\beta^2\gamma-\alpha|/|\alpha|
=|\alpha\beta^2\gamma-\alpha^2|/\alpha^2
=(\alpha^2-\alpha\beta^2\gamma)/\alpha^2=
(\alpha-\beta^2\gamma)/\alpha$. 

\emph{Case~1(c)($-$):} $p<0$, $\Delta<0$, and $\alpha^2-\alpha\beta^2\gamma\leq 0$.\\ 
Then 
$-3\alpha^2\geq -\alpha^2-2\alpha\beta^2\gamma$ 
and therefore $x_+=-1$. 
It follows that our desired root is the ``largest root'' corresponding to 
$k=0$ in \cref{t:genroots}\cref{t:genroots1c}:
\begin{align*}
x
&=x_0+2(-p/3)^{1/2}\cos\bigg(\frac{1}{3}\Big(\arccos \frac{-q/2}{(-p/3)^{3/2}}\Big) \bigg)\\
&=
-\frac{2\alpha+\beta^2\gamma}{3\alpha}
+\frac{2|\beta^2\gamma-\alpha|}{3|\alpha|}\cos\bigg(\frac{1}{3}\Big(\arccos \frac{-q/2}{(-p/3)^{3/2}}\Big) \bigg)\\
&=
-\frac{2\alpha+\beta^2\gamma}{3\alpha}
+\frac{2(\beta^2\gamma-\alpha)}{3\alpha}\cos\bigg(\frac{1}{3}\Big(\arccos \frac{-q/2}{(-p/3)^{3/2}}\Big) \bigg),
\end{align*}
where in the last line we used the assumption to deduce that 
$|\beta^2\gamma-\alpha|/|\alpha|
=|\alpha\beta^2\gamma-\alpha^2|/\alpha^2
=(\alpha\beta^2\gamma-\alpha^2)/\alpha^2=
(\beta^2\gamma-\alpha)/\alpha$. 

Note that the last two cases can be combined to obtain \cref{e:xDneg11}.

\noindent
\emph{Case~2:} $p=0$, i.e., $\alpha-\beta^2\gamma= 0$ by \cref{e:230111b}.
Then $\Delta = (q/2)^2\geq 0$; hence, $\sqrt{\Delta}=|q|/2$ and 
thus $\{-q/2\pm\sqrt{\Delta}\} = \{-q,0\}$.
By \cref{t:genroots}\cref{t:genroots2}, the only real root is 
$x_0+(-q)^{1/3}=x_0+(-q/2+\sqrt{\Delta})^{1/3} + (-q/2-\sqrt{\Delta})^{1/3}$ which is the same as \cref{e:xDnonneg11} using \cref{e:230111d}.

\noindent
\emph{Case~3:} $p>0$. In view of \cref{e:230111b}, this case never occurs.
\end{proof}

\section{A proximal mapping of a closure of a perspective function}

\label{sec:perspective}

The following completes \cite[Example~24.57]{BC2017} which stopped short of providing solutions for a cubic encountered.
\begin{example}
Define the function $h$ on
$\RR^n\times\RR$ by 
\begin{equation}
h(y,\eta) := 
\begin{cases}
\|y\|^2/(2\eta), &\text{if $\eta>0$;}\\
0, &\text{if $y=0$ and $\eta=0$;}\\
\pinf, &\text{otherwise.}
\end{cases}
\end{equation}
Let $\gamma>0$ and $(y,\eta)\in\RR^n\times\RR$. 
Then 
\begin{equation}
\prox_{\gamma h}(y,\eta) = 
\begin{cases}
(0,0), &\text{if $\|y\|^2+2\gamma\eta\leq 0$;}\\
\Big(\big(1-\frac{\gamma\lambda}{\|y\|} \big)y,\eta+\frac{\gamma\lambda^2}{2} \Big), &\text{if $\|y\|^2+2\gamma\eta>0$,}
\end{cases}
\end{equation}
where 
$p = 2(\eta+\gamma)/\gamma$,
$\Delta = (p/3)^3+(\|y\|/\gamma)^2$, and 
\begin{equation}
\label{e:230106}
\lambda = \begin{cases}
\sqrt[\mathlarger 3]{\dfrac{\|y\|}{\gamma}+\sqrt{\Delta}}
+
\sqrt[\mathlarger 3]{\dfrac{\|y\|}{\gamma}-\sqrt{\Delta}}, &\text{if $\Delta\geq 0$;}\\[+7mm]
2(-p/3)^{1/2}\cos\Big(\dfrac{1}{3}\arccos\dfrac{\|y\|/\gamma}{(-p/3)^{3/2}}\Big), &\text{if $\Delta<0$.}
\end{cases}
\end{equation}
\end{example}
\begin{proof}
The first cases where already provided in \cite[Example~24.57]{BC2017}.
Now assume 
$\|y\|^2+2\gamma\eta>0$. 
It was also observed in \cite[Example~24.57]{BC2017} that if $y=0$, then $\prox_{\gamma h}(y,\eta)=(0,\eta)$. 

So assume also that $y\neq 0$.
It follows from the discussion that in 
\cite[Example~24.57]{BC2017} that 
$\lambda$ is the unique positive solution of 
the already depressed cubic
\begin{equation}
\lambda^3 + \frac{2(\eta+\gamma)}{\gamma}\lambda
-\frac{2\|y\|}{\gamma}=0,
\end{equation}
which is where the discussion in \cite{BC2017} halted. Continuing here, we set
\begin{equation}
p := \frac{2(\eta+\gamma)}{\gamma},
\quad 
q := -\frac{2\|y\|}{\gamma} < 0,
\end{equation}
and $\Delta := (p/3)^3+(q/2)^2$. 
Using \cref{c:roots}, we see that if $\Delta<0$, then 
\begin{equation}
\lambda = 2(-p/3)^{1/2}\cos\Big(\frac{1}{3}\arccos\frac{-q/2}{(-p/3)^{3/2}}\Big)
\end{equation}
while if $\Delta\geq 0$, then 
\begin{equation}
\lambda = 
\sqrt[\mathlarger 3]{\frac{-q}{2}+\sqrt{\Delta}}
+
\sqrt[\mathlarger 3]{\frac{-q}{2}-\sqrt{\Delta}}
\end{equation}
which slightly simplifies to the expression provided in \cref{e:230106}.

Finally, notice that if $y=0$, then 
the assumption that $\|y\|^2+2\gamma\eta>0$ yields 
$\eta>0$; thus, $p>0$, $q=0$, and hence $\Delta>0$. Formally, our $\lambda$ then simplifies to $0$ which conveniently allows us to combine this case with the case $y\neq 0$. 
\end{proof}

\section*{Acknowledgments}
The authors thank Dr.\ Amy Wiebe for referring us to \cite{BPT}. 
HHB is supported by the Natural Sciences and
Engineering Research Council of Canada. MKL was partially supported by SERB-UBC fellowship and NSERC Discovery grants of HHB and XW.


\end{document}